\documentclass[letterpaper,11pt,reqno]{amsart}
\usepackage[margin=1in]{geometry}
\usepackage{mathrsfs,xfrac} 
\usepackage[colorlinks=true,linkcolor=blue,citecolor=blue,urlcolor=blue]{hyperref} 
\usepackage{amsmath,amssymb,amsthm,amsfonts,amsbsy,latexsym,dsfont,color} 
\usepackage[foot]{amsaddr}
\usepackage{mathtools,microtype,needspace,booktabs}

\usepackage[T1]{fontenc}
\usepackage[utf8]{inputenc}
\usepackage{comment}

\usepackage[textsize=tiny]{todonotes}
\usepackage{regexpatch}
\makeatletter
\xpatchcmd{\@todo}{\setkeys{todonotes}{#1}}{\setkeys{todonotes}{inline,#1}}{}{}
\makeatother

\usepackage{enumitem}

\newenvironment{enumeratea}{\begin{enumerate}[label=\textup{\alph*)},leftmargin=2em,itemsep=1ex]}{\end{enumerate}}

\allowdisplaybreaks

\newtheorem{thm}{Theorem}[section]
\newtheorem{lem}[thm]{Lemma}
\newtheorem{cor}[thm]{Corollary}
\newtheorem{prop}[thm]{Proposition}

\theoremstyle{definition}

\newtheorem{rem}[thm]{Remark}

\renewcommand{\le}{\leqslant} 
\renewcommand{\ge}{\geqslant} 
\renewcommand{\leq}{\leqslant} 
\renewcommand{\geq}{\geqslant}

\newcommand{\ind}{\mathds{1}}
\newcommand{\eps}{\varepsilon}

\newcommand{\norm}[1]{\left\Vert#1\right\Vert}
\newcommand{\abs}[1]{\left\vert#1\right\vert}

\def\qed{ \hfill $\blacksquare$}  
\let\ga=\alpha \let\gb=\beta \let\gc=\gamma  
     \let\gl=\lambda           \let\gs=\sigma  
  
 \let\gD=\Delta  \let\gL=\Lambda

\newcommand{\cB}{\mathcal{B}}
\newcommand{\cD}{\mathcal{D}}\newcommand{\cE}{\mathcal{E}}\newcommand{\cF}{\mathcal{F}}
\newcommand{\cG}{\mathcal{G}}\newcommand{\cH}{\mathcal{H}}
\newcommand{\cK}{\mathcal{K}}

\newcommand{\cV}{\mathcal{V}}
\newcommand{\cZ}{\mathcal{Z}}  
\newcommand{\vone}{\mathbf{1}}

\newcommand{\sP}{\mathscr{P}}

\DeclareMathOperator{\E}{\mathds{E}}
\DeclareMathOperator{\pr}{\mathds{P}}

\DeclareMathOperator{\tr}{Tr} 
  
\DeclareMathOperator{\diag}{diag}

\DeclareMathOperator{\geom}{geom}
\DeclareMathOperator{\dist}{dist}

\DeclareMathOperator{\cen}{cent}

\title[Eigenvector Preferential Attachment Trees]{Persistent Hubs and Sharp Asymptotics in Eigenvector Preferential Attachment Trees}
\author[Dey]{Partha S.~Dey}
\address{Department of Mathematics, University of Illinois Urbana-Champaign}
\email{psdey@illinois.edu}
\date{\today}

\begin{document}
\begin{abstract}
	We prove sharp almost-sure asymptotics for eigenvector preferential
	attachment trees in which each attachment uses eigenvector centrality
	with probability $\ga\in(0,1]$ and uniform selection with probability
	$1-\ga$. A unique persistent hub eventually has the maximum degree,
	which is asymptotic to $\gc^2t^{2/3}$, where
	$\gc=(9\ga/10)^{1/3}$. The Perron--Frobenius eigenvalue is
	asymptotic to $\gc t^{1/3}$, and the second-largest degree is
	asymptotic to $2\gc t^{1/3}$. The hub and its neighbors each carry
	asymptotically half of the squared Perron mass. We also identify the
	centrality profile by degree and an explicit mixture of geometric laws
	for the degree of a uniformly sampled vertex. A limiting distance profile
	quantifies the dandelion geometry, and vertices beyond the hub's
	neighbors have maximum degree of order at most $\log t$.
\end{abstract}

\maketitle

\begin{figure}[hbtp]
	\centering
    \vspace*{-1cm}
	\setlength{\fboxsep}{-0pt}
	\begin{minipage}[c]{.42\linewidth}
		\includegraphics[width=\linewidth]{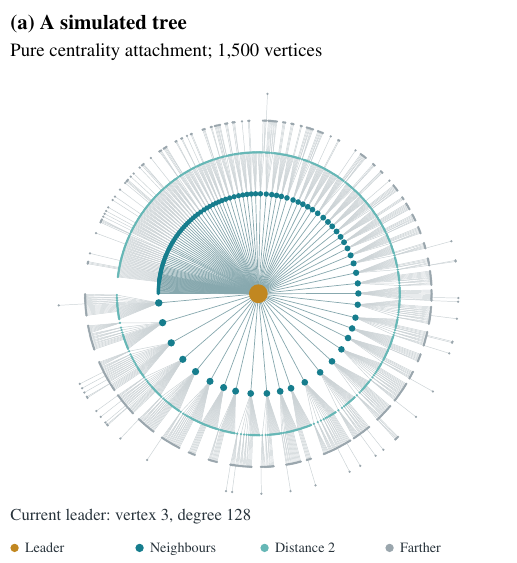}
	\end{minipage}
	\begin{minipage}[c]{.48\linewidth}
		\includegraphics[width=\linewidth]{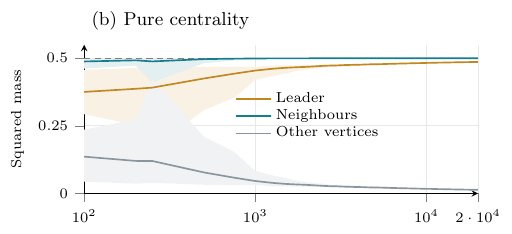}\\[-4pt]
		\includegraphics[width=\linewidth]{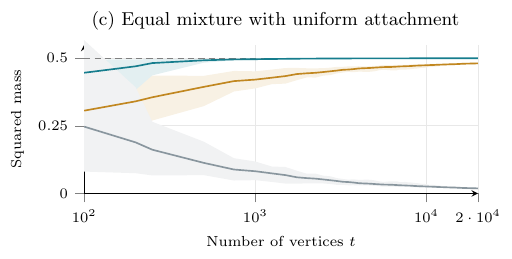}
	\end{minipage}
	\caption{
		(a) A simulated $1500$-vertex tree with $\ga=1$, colored by distance from its current degree leader.
		(b)--(c) Squared Perron mass on the current leader, its neighbors, and the remaining vertices for $\ga=1$ and $\ga=1/2$, respectively. Curves and shading show means and observed ranges of $12$ runs per regime; the shading is not a confidence interval. The predicted limits are $1/2$, $1/2$, and $0$. 
		}
	\label{fig:simulation}
\end{figure}

\section{Introduction}\label{sec:intro}
Classical preferential attachment favors vertices of high degree~\cite{BarabasiAlbert1999}, while eigenvector centrality rewards vertices whose neighbors are themselves central~\cite{Bonacich1972}. Adami, Ebadi, and Nattagh-Najafi~\cite{Adami2024} introduced an attachment rule based on eigenvector centrality and observed a \emph{dandelion} structure numerically. Their mean-field analysis predicts degree growth $t^{2/3}$ for the hub and $t^{1/3}$ for its neighbors. For a single-edge seed and pure centrality attachment, our model is their tree case. We prove almost-sure persistence and sharp asymptotics, allow arbitrary fixed tree seeds and a fixed uniform component, and identify an explicit limiting degree law together with its missing first moment. The general $m\ge 2$ case will be analyzed in an upcoming paper.

The dependence on the whole tree is the main obstacle. Vertices of the same degree can have different centrality scores, and even an attachment elsewhere changes their scores. Thus persistence results for degree-based rules~\cite{Galashin2016,BanerjeeBhamidi2021} and the branching-process approach to local attachment rules~\cite{RudasTothValko2007} do not apply directly. Persistence of other centrality notions has also been studied. Jog and Loh~\cite{JogLoh2018} prove persistence for centroid, or Jordan, centrality in growing trees, where the score of a vertex is the largest component size after deleting that vertex. Banerjee and Bhamidi~\cite{BanerjeeBhamidi2022} establish persistence of the optimal $K$ Jordan centers under general conditions on the degree-based attachment function. This centrality notion differs from the Perron score that drives the present dynamics.

Eigenvector localization in networks is a well-established phenomenon~\cite{MartinZhangNewman2014}. Here it emerges jointly with the growth dynamics --- a persistent hub carries half the squared Perron mass, although its normalized attachment weight tends to zero. The other half is spread over its neighbors. Our finite-core competition argument proves that the global reinforcement selects a single surviving hub. The subsequent shell balances determine the sharp constants and typical-vertex laws.

\subsection{Model}\label{ssec:model}
We start from a fixed connected seed tree on $[t_0]$, where $t_0\geq2$ is fixed.
For the tree $T_t$, let $\gl_t$ be its Perron--Frobenius
eigenvalue. Let $v(t)=(v_u(t):u\in[t])$ be its positive eigenvector
normalized by $\sum_{u=1}^t v_u(t)^2=1$. Define
$S_t:=\sum_{u=1}^t v_u(t)$. We fix 
\begin{align*}
\ga\in(0,1]\text{ and define }\gb:=1-\ga.
\end{align*}
At time $(t+1)$, attach the new vertex to the parent $J_{t+1}\in[t]$
with probability
\begin{align*}
	a_t(u)=\ga\cdot \frac{v_u(t)}{S_t}+\gb\cdot \frac{1}{t}.
\end{align*}
Let $d_u(t)$ be the degree of vertex $u$ in $T_t$, and let
$D^\star_t$ and $D^{\star\star}_t$ be the largest and second-largest
degrees, counting multiplicity. Throughout, define the constants
\begin{align*}
	\theta=2/3,\quad
	\gc
	=\left(9\ga/{10}\right)^{1/3}.
\end{align*}
For $b>0$, we write $\geom_b$ for the geometric probability mass function
$
	\geom_b(k)=b^{k-1}/(1+b)^{k}, k\geq1,
$
and define $\geom_0(k)=\ind_{\{k=1\}}$. This distribution is supported on
$\{1,2,\ldots\}$ and has mean $1+b$.

\subsection{Main result}\label{ssec:mainres}

For $k\geq1$, let $N_k(t)$ count the degree-$k$ vertices of $T_t$.
The following theorem collects the persistence, growth, localization,
and degree-distribution conclusions proved in Sections~\ref{sec:competition}
and~\ref{sec:sharp}.

\begin{thm}\label{thm:main}
Almost surely, the following conclusions hold.
\begin{enumeratea}
\item\label{main:a} \emph{(Persistence and growth.)} There is a unique random vertex $h$ that is the unique degree maximizer for all sufficiently large $t$, and as $t\to\infty$,
\begin{align*}
 \frac{D^\star_t}{\gl_t^2}\to1,\qquad
 \frac{D^{\star\star}_t}{2\gl_t}\to1,\qquad
 \frac{\gl_t}{\gc t^{1/3}}\to1.
\end{align*}
\item\label{main:b}  \emph{(Spectral localization and centrality.)} With
$\kappa_t=S_t^{-1}\sum_u v_u(t)^3$, as $t\to\infty$
\begin{align*}
 \frac{S_t}{\gl_t}\to\frac{1+\theta}{\sqrt2},\qquad
 \gl_t\kappa_t\to\frac1{2(1+\theta)},\qquad
 v_h(t)^2\to\frac12,\qquad
 \sum_{u\sim h}v_u(t)^2\to\frac12.
\end{align*}
For every fixed $k\geq1$,
\begin{align*}
A_{k}(t):= \frac1{S_t}\sum_{u:d_u(t)=k}v_u(t)
 \to \frac{\theta}{1+\theta}\geom_{\gb+\theta}(k).
\end{align*}
\item\label{main:c}  \emph{(Typical degree.)} For every fixed $k\geq1$,
\begin{align*}
 \frac1tN_k(t)\to \pi_k
 :=\frac{\ga}{1+\theta}\geom_{\gb+\theta}(k)
   +\frac{\gb+\theta}{1+\theta}\geom_\gb(k).
\end{align*}
The conditional law, given $T_t$, of the degree of an independent uniform vertex converges almost surely in total variation to $\pi$. Its unconditional law has the same limit.
\end{enumeratea}
\end{thm}

\begin{proof}[Proof of Theorem~\ref{thm:main}]
Proposition~\ref{prop:scales} gives the matching growth scales.
Theorems~\ref{thm:hub},~\ref{thm:sharp}, and~\ref{thm:second} yield part~\ref{main:a};
Theorems~\ref{thm:sharp} and~\ref{thm:profile} give part~\ref{main:b};
Theorem~\ref{thm:degree} gives part~\ref{main:c}.
\end{proof}

The hub's normalized centrality $v_h(t)/S_t$ tends to zero despite
$v_h(t)^2\to1/2$. Vertices at distances one and two from the hub carry
normalized centrality converging to $1/(1+\theta)$ and
$\theta/(1+\theta)$, respectively.
These are fractions of the $\ell_1$ mass, whereas the localization
statement in Theorem~\ref{thm:main} concerns squared $\ell_2$ mass.
The mean of $\pi$ is $2-\ga/(1+\theta)$, whereas the mean degree of $T_t$ tends to $2$. The missing first moment is carried collectively by the vanishing fraction of vertices adjacent to the hub. The hub itself has sublinear degree, and it receives a vanishing fraction of all edges. Corollaries~\ref{cor:distance} and~\ref{cor:hierarchy} further describe the geometry. They give a limiting distance profile and show that degrees beyond the hub's neighbors are $O(\log t)$.

\begin{rem}[Pure centrality attachment]
When $\ga=1$, the typical degree law becomes
\begin{align*}
 \pi_1=\frac{19}{25},\quad
 \pi_k=\frac9{25}\left(\frac25\right)^{k-1}\text{ for } k\geq2,\text{ with mean }
 \sum_{k\geq1}k\pi_k=\frac75.
\end{align*}
The missing first moment is $3/5$, carried by the hub's neighbors as described in the proof.
\end{rem}

In this article, all results concern fixed $\ga>0$. At $\ga=0$ the model is uniform attachment, which is outside the scope of the growth and persistence assertions. Although the formula for $\pi_k$ has the limiting value $2^{-k}$, we do not prove estimates uniform in $\ga\downarrow0$.

Figure~\ref{fig:simulation} uses $12$ runs from a single edge for each of $\ga=1$ and $\ga=1/2$, up to time $20000$. All $24$ degree leaders were unique and unchanged from $t=631$ through $20000$. These finite simulations
illustrate the hub picture. Eventual persistence is established by the proof below.

\subsection{Heuristics and proof strategy}\label{ssec:heu}

The following scaling argument motivates the exponent $\theta=2/3$ fixed above. 
First, we recall that for a star graph with hub degree $d$, the Perron-Frobenius eigenvalue is $\sqrt{d}$ with unit eigenvector $(1/\sqrt 2, 1/\sqrt{2d},\ldots, 1/\sqrt{2d})$. Its $\ell_1$ norm is $(1+\sqrt d)/\sqrt2$. In particular, the squared Perron mass of the hub is $1/2$, and of its neighbors is also $1/2$. More generally, the unit positive Perron vector of any connected
bipartite graph has squared mass $1/2$ on each of its two parts.

Now, for the eigenvector preferential attachment tree model, if the hub degree $H$ has scale $t^\vartheta$,
the star-like geometry to be proved suggests that $\gl$ and $S$
have order $\sqrt H$, while the hub coordinate stays of constant
order. Its centrality attachment probability is then of order
$H^{-1/2}\asymp t^{-\vartheta/2}$. Summing over time gives the degree
scale $t^{1-\vartheta/2}$; the uniform component contributes only
logarithmically. Matching powers gives
$
	\vartheta=1-{\vartheta}/{2},
	\text{ hence } \vartheta=\theta=2/3.
$
A fixed neighbor of the hub has an attachment rate of order
$H^{-1}\asymp t^{-\theta}$, giving degree scale
$t^{1-\theta}=t^{\theta/2}$.

Births into the second shell occur at asymptotically $\gl$ times the
hub attachment rate. A hub attachment increases $H^{1/\theta}$
asymptotically by $\gl/\theta$. Comparing these drifts and controlling
martingale errors gives second shell size
$\sim\theta H^{1/\theta}\sim\theta\gl^3$. The first two shells contribute asymptotically $1/(1+\theta)$ and
$\theta/(1+\theta)$ of $S$, which gives hub attachment rate
$\ga/((1+\theta)\gl)$. The drift of $H^{1/\theta}$ then yields
\begin{align*}
	S\sim\frac{1+\theta}{\sqrt2}\gl,\quad
	\gl^3\sim\frac{\ga t}{\theta(1+\theta)}.
\end{align*}
Section~\ref{sec:sharp} makes these balances rigorous, and
Remark~\ref{rem:constants} summarizes the resulting constants.

The first hurdle is to establish this geometry without assuming a hub.
Using spectral and moment estimates, we first obtain polynomial upper and lower bounds for the maximum degree $D^\star_t$ and Perron eigenvalue $\gl_t$, and confine vertices with degrees of order $\gl^2$
to a finite old core. Until any exterior degree reaches a small fraction
of $\gl^2$, late vertices have much smaller centrality than their
core ancestors. An exponential comparison controls the first crossing.

Several vertices in the core may still compete for dominance.
Surviving competitors have equal leading degrees, and outward drift
confines lasting near ties to a window of width $O(\gl)$.
Eliminating all other vertices gives a finite symmetric matrix whose centered row sums measure the competitors' imbalance. Before exit from fixed deterministic bounds, Perron reinforcement and attachment noise make the conditional drift of their squared norm at least a positive constant times $1/t$. Permanent coexistence would contradict boundedness of the stopped
statistic because $\sum_t1/t$ diverges.
Applying this argument to all finite candidate sets leaves one persistent hub.

Uniform shell estimates then give the sharp second-largest degree.
The selection rates $(1+\theta-\ga+o(1))/t$ in the second shell
and $(1-\ga+o(1))/t$ farther away yield the geometric mixture through
linear degree-count recursions.

\subsection{Notation and probability conventions}
In addition to the notation above, define $M_t:=\max_u v_u(t)$.
Define $\geom_b(0)=N_0(t)=A_0(t)=0$. We also use the constants
\begin{align*}
	q={\ga}/{(1+\theta)},\quad b_2=\gb+\theta.
\end{align*}
Write $\cF_t=\gs(T_{t_0},J_{t_0+1},\ldots,J_t)$ for the graph
history and $\E_t=\E(\,\cdot\mid\cF_t)$. For a process $X_t$,
write $\gD X_t=X_{t+1}-X_t$. Unless a different lower limit is
specified, $\sum_{s<t}$ means $\sum_{s=t_0}^{t-1}$.
Here, $x_t\asymp y_t$ (equivalently, $x_t=\Theta(y_t)$) means that
$x_t/y_t$ is eventually bounded above and below by positive constants.
$x_t\gtrsim y_t$ (equivalently, $x_t=\Omega(y_t)$) means an eventual
positive lower bound.
An eventual bound with random constants is always used by fixing a
deterministic finite label set, a starting time, and rational constants,
and stopping at the first failure of the corresponding adapted bounds.
Martingale estimates are proved for these stopped processes. Only then
do we take the countable union of events on which no failure occurs.
Transition laws are never conditioned on a future successful-core or
hub event. All growth bounds below refer to almost-sure paths. Their constants may be random until localized and may depend on $\ga$. We use the following elementary martingale normalization repeatedly, which follows from martingale convergence and Kronecker's lemma~\cite[Chapters~11--12]{Williams1991}.

\begin{lem}\label{lem:martingale-normalization}
Let $(\xi_{t+1})$ be a square-integrable martingale difference sequence adapted to $(\cF_t)$, and let $b_t>0$ be deterministic, nondecreasing, and tend to infinity. If
$
 \sum_{t\geq t_0} {\E\xi_{t+1}^2}/{b_t^2}<\infty,
$
then $\sum_{s<t}\xi_{s+1}=o(b_t)$ almost surely.
\end{lem}
The quantities $v_u(t)/S_t$ and $A_k(t)$ always refer to normalized
$\ell_1$ centrality. Squared Perron mass means a sum of $v_u(t)^2$.

\subsection{Organization of the paper}
Section~\ref{sec:growth} proves the initial upper and lower bounds without
assuming a hub. Section~\ref{sec:core} constructs a finite old core and
upgrades these bounds to matching growth scales in
Proposition~\ref{prop:scales}. Section~\ref{sec:competition} resolves the
competition within that core and proves the persistent-hub theorem.
Section~\ref{sec:sharp} derives the sharp constants, centrality profile,
uniform-vertex degree law, and second-largest-degree asymptotic, and
proves the distance-profile and degree-hierarchy corollaries.

\subsection*{Declaration on the use of AI}
AI tools were used at the final revision stage to assist with proofreading, grammar checking, and reviewing the mathematical arguments.  All AI suggestions incorporated into the manuscript were reviewed by the author, who takes full responsibility for the content of the paper.

\subsection*{Acknowledgments.} The author thanks Neeladri Maitra and Jaya M.~Manthripragada for many enlightening discussions.

\section{Spectral and initial growth estimates}\label{sec:growth}

We first state and prove several classical spectral estimates for trees. We use these results to obtain initial upper and lower bounds for the maximum degree in our model. These first estimates hold almost surely and do not assume localization or persistence. Proposition~\ref{prop:upper} gives $D^\star_t=O(t^{2/3})$. Hence, $\gl_t=O(t^{1/3})$ by Lemma~\ref{lem:spectral}. This upper bound makes the predictable
martingale gates used below eventually equal to one. On the gates,
conditional variances have deterministic bounds that become summable
after the indicated normalization.

Proposition~\ref{prop:bootstrap} supplies the initial lower bound
$\gl_t^2=\Omega(t^{2/7})$. This exponent is sufficient because
$2/7>1/4$, so $\gl_t^2\gg t^{1/4}$ eventually, even when the
positive comparison constant is random. At a first crossing of the
exterior-degree barrier $\eps\gl_t^2$, the exponential estimate in
Proposition~\ref{prop:finite-core} then has a summable error probability.
This constructs the finite old core. Its positive Perron mass lets
Proposition~\ref{prop:scales} improve the initial bounds to
$D^\star_t=\Theta(t^{2/3})$ and
$\gl_t,S_t=\Theta(t^{1/3})$. These matching scales are established
before the competition argument and supply its growth and variance
estimates.

\begin{lem}\label{lem:spectral}
	For every tree on at least three vertices, write $d_u$ for its vertex
	degrees. Let $v$ be its unit positive Perron vector, $\gl$ its
	Perron eigenvalue, $D^\star$ its maximum degree, and $S=\sum_u v_u$. Then, we have
	\begin{align*}
		\gl S=\sum_u d_uv_u,\quad \sqrt{D^\star}\leq\gl
		\leq
        2\sqrt{D^\star-1},
        \quad \max_u v_u\leq2^{-1/2}.
	\end{align*}
	With $\kappa=S^{-1}\sum_u v_u^3$, we have $\kappa\geq1/S^2$ and
	\begin{align}\label{eq:kappa-tree}
		\kappa\geq\frac{\gl}{4S^2}
		\geq\frac{\gl^3}{4Q},\quad Q=\sum_u d_u^2.
	\end{align}
\end{lem}
\begin{proof}
	Summing the coordinates of $Av=\gl v$ gives the first identity. If $h$ has maximum degree, then
	$\sqrt{D^\star}=\norm{Ae_h}\leq\norm A=\gl$, which gives the lower bound. For the upper bound,
	root the tree and define $b=\sqrt{D^\star-1}$. On an oriented edge $p\to c$,
	$2\abs{x_px_c}\leq x_p^2/b+bx_c^2$. Summing over all oriented edges gives at most $2b\sum_u x_u^2$. Here, we used the fact that a nonroot has at most $b^2$ children, and the root
	coefficient is $D^\star/b\leq2b$. The Rayleigh principle gives the upper bound of $2b$.
	At any vertex $u$, the Cauchy--Schwarz inequality gives
	$\sum_{w\sim u}v_w^2\geq\gl^2v_u^2/d_u\geq v_u^2$.
	Normalization gives $2v_u^2\leq1$.

	The Cauchy--Schwarz inequality also gives
	$1=\sum_u v_u^{1/2}v_u^{3/2}\leq(S\sum_u v_u^3)^{1/2}$.
	For the stronger tree inequality, orient edges from parent to child to obtain
	\begin{align*}
		\frac{\gl^2}{4}
		=\left(\sum_{p\to c}v_pv_c\right)^2
		\leq\left(\sum_{p\to c}v_p^2v_c\right)
		\left(\sum_{p\to c}v_c\right)
		\leq\gl S\sum_p v_p^3.
	\end{align*}
	Finally, $\gl S\leq\sqrt Q$, proving~\eqref{eq:kappa-tree}.
	We will also use the forest consequence $\norm{A_F}\leq2\sqrt{\gD_F}$,
	where $\gD_F$ is its maximum degree. Apply the tree bound to each component with at least three vertices; a single edge has norm $1$, and an isolated vertex has norm $0$.
\end{proof}

\begin{lem}[Eigenvalue update in one step]\label{lem:pendant}
	If a new leaf is attached at $u$ in a connected bipartite graph
	on at least two vertices, then
	$
		2v_u^2\leq\gl'^2-\gl^2\leq1,
	$
    where $\gl,\gl'$ are the Perron-Frobenius eigenvalues of the old and new graph, respectively.
	Consequently, in the growth process, we have
	\begin{align}\label{eq:cubic-increment}
		3\gl_t v_{J_{t+1}}(t)^2
		\leq\gl_{t+1}^3-\gl_t^3
		\leq\tfrac32\sqrt{\gl_t^2+1}.
	\end{align}
\end{lem}
\begin{proof}
	Choose the bipartition with $u$ in its first part and write the
	adjacency matrix as $A=\bigl(\begin{smallmatrix}0&B\\B^\top&0\end{smallmatrix}\bigr)$. Thus $AA^\top = \bigl(\begin{smallmatrix}BB^\top&0\\0&B^\top B\end{smallmatrix}\bigr)$.
	The eigenvector equations give squared Perron mass $1/2$ on each part.
	Adding the leaf replaces $BB^\top$ by $BB^\top+e_ue_u^\top$.
	Rayleigh testing at the unit vector $\sqrt2\cdot (v$ restricted to the first
	part) gives the lower bound $\gl'^2\ge \gl^2+2v_u^2$. The operator norm of the update gives
	the upper bound. Finally, we apply the fact that $(x+a)^{3/2}-x^{3/2}=\frac32 \int_0^a \sqrt{x+u}\,du$ to obtain the last display.
\end{proof}

\begin{prop}[Initial upper bound for the maximum degree]\label{prop:upper}
	Almost surely, $\limsup_t (D^\star_t)^{3/2}/t\leq 3$.
\end{prop}
\begin{proof}
	Only attachment to a current degree maximizer increases $D^\star_t$. The
	spectral identity and $\sum_u d_u(t)=2(t-1)$ give
	\begin{align*}
		\pr(D^\star_{t+1}=D^\star_t+1\mid\cF_t)
		\leq\frac{\ga\gl_t+2\gb(1-1/t)}{D^\star_t}
		\leq\frac2{\sqrt{D^\star_t}}, \quad t\geq3.
	\end{align*}
	Define $W_t=(D^\star_t)^{3/2}$. Then
	$\E_t\gD W_t\leq3\sqrt{1+1/D^\star_t}$, and
	$\E_t(\gD W_t)^2\leq9\sqrt{D^\star_t}\leq9\sqrt t$.
	For the Doob martingale $L_t=W_t-W_{t_0}-\sum_{s<t}\E_s\gD W_s$,
	the series $\sum_s\gD L_s/s$ converges almost surely and in $L^2$,
	because its total variance is at most $C\sum_s s^{-3/2}<\infty$.
	Kronecker's lemma gives $L_t=o(t)$. If $D^\star_t$ stays bounded, the result is immediate. Otherwise, $D^\star_t\uparrow\infty$ and the Ces\`aro mean of
	the drift bound tends to $3$.
\end{proof}

\begin{prop}[Initial polynomial lower bound]\label{prop:bootstrap}
	Almost surely, $\gl_t=\Omega(t^{1/7})$.
\end{prop}
\begin{proof}
	Write
	\begin{align*}
		r_t=\gl_t^2,\qquad Q_t=\sum_u d_u(t)^2,\qquad
		x_t=\gD r_t,\qquad y_t=\gD Q_t.
	\end{align*}
	We first compare $Q_t$ with a power of $r_t$, then compare $Q_t$
	with its compensator. These two comparisons yield the lower bound.

\smallskip\noindent\textbf{Step 1: Drift estimates and divergence of $r_t$.}
	Lemma~\ref{lem:pendant}, Lemma~\ref{lem:spectral}, and the attachment
	rule give
	\begin{align}
		&0\leq x_t\leq1,\qquad
		\E_t x_t\geq2\ga\kappa_t+2\gb/t
		\geq\frac{\ga\gl_t^3}{2Q_t},\label{eq:rho-drift}\\
		&y_t=2d_{J_{t+1}}(t)+2,\qquad
		b_t:=\E_t y_t=2\ga\gl_t+2\ga+6\gb-4\gb/t.
		\label{eq:Q-drift}
	\end{align}
	The formula for $b_t$ uses
	$\E_t d_{J_{t+1}}(t)=\ga\gl_t+2\gb(t-1)/t$.
	Also, $\kappa_t\geq1/S_t^2\geq1/t$, so $\E_t x_t\geq2/t$.
	For $c=1-e^{-1}$, the inequality
	$e^{-x}\leq1-cx$ on $[0,1]$ therefore gives
	\begin{align*}
		\E_t e^{-r_{t+1}}\leq e^{-r_t}(1-2c/t).
	\end{align*}
	Iteration shows $\E e^{-r_t}\to0$. Since $r_t$ is nondecreasing,
	bounded convergence implies $r_t\to\infty$ almost surely.

\smallskip\noindent\textbf{Step 2: A supermartingale comparison for $Q_t$.}
	Choose a deterministic $K>8+24\gb/\ga$ and put
	$f(r)=r^{-4}e^{K/\sqrt r}$. Direct differentiation gives
	\begin{align*}
		-\frac{f'(r)}{f(r)}=\frac4r+\frac{K}{2r^{3/2}},\qquad
		\frac{f''(r)}{f(r)}
		=\frac{20}{r^2}+\frac{19K}{4r^{5/2}}+\frac{K^2}{4r^3}.
	\end{align*}
	Both $f$ and $f''$ are decreasing. Taylor's theorem and $x^2\leq x$
	thus imply, for $r\geq1$ and $0\leq x\leq1$,
	\begin{align*}
		f(r+x)-f(r)
		\leq-f(r)\left(\frac4r+\frac{K}{2r^{3/2}}
		-\frac{C_K}{r^2}\right)x,
	\end{align*}
	where $C_K$ is a finite deterministic constant.
	Define, $R_t=Q_tf(r_t)$. Since $y_t\geq0$ and $f$ is decreasing,
	\begin{align*}
		\gD R_t
		&=Q_t\bigl(f(r_t+x_t)-f(r_t)\bigr)+y_tf(r_t+x_t)\\
		&\leq Q_t\bigl(f(r_t+x_t)-f(r_t)\bigr)+y_tf(r_t).
	\end{align*}
	For all sufficiently large $r_t$, the coefficient of $x_t$ in
	the Taylor bound is negative. Applying~\eqref{eq:rho-drift}
	and~\eqref{eq:Q-drift} then gives
	\begin{align*}
		\E_t\gD R_t
		\leq f(r_t)\left(2\ga+6\gb-\frac{K\ga}{4}
		+\frac{C_K\ga}{2\sqrt{r_t}}-\frac{4\gb}{t}\right)
		\leq0.
	\end{align*}
	The last inequality holds above a deterministic threshold $r_0$,
	because $2\ga+6\gb-K\ga/4<0$.

	To justify this comparison without conditioning on future growth,
	fix a deterministic $n$ and let $G_n=\{r_n\geq r_0\}\in\cF_n$.
	Monotonicity of $r_t$ shows that
	$(\ind_{G_n}R_t)_{t\geq n}$ is a nonnegative supermartingale.
	Its initial value is integrable: at the fixed time $n$ there are
	only finitely many possible trees, and $r_n\geq1$.
	It therefore converges and is bounded almost surely.
	Since $r_t\to\infty$, almost every path belongs to some $G_n$.
	Taking the countable union over $n$, and using
	$f(r)^{-1}=r^4e^{-K/\sqrt r}\leq r^4$, yields
	\begin{align}\label{eq:Q-rho}
		Q_t=O(r_t^4)=O(\gl_t^8)\quad\hbox{almost surely}.
	\end{align}

\smallskip\noindent\textbf{Step 3: Control of the compensator error.}
	Let
	\begin{align*}
		L_t=Q_t-Q_{t_0}-\sum_{s<t}b_s.
	\end{align*}
	This is a martingale, with $\gD L_t=y_t-b_t$.
	Since every degree is at least one and
	$D^\star_t\leq\gl_t^2$, we have
	\begin{align*}
		y_t^2&\leq16d_{J_{t+1}}(t)^2,\\
		\E_t d_{J_{t+1}}(t)^2
		&\leq D^\star_t\E_t d_{J_{t+1}}(t)
		\leq\gl_t^2(\ga\gl_t+2\gb).
	\end{align*}
	Consequently $\E_t(\gD L_t)^2\leq\E_t y_t^2\leq C\gl_t^3$.
	Use the predictable gate
	\begin{align*}
		\chi_t=\ind_{\{\gl_t\leq5t^{1/3}\}},\qquad
		\widetilde L_t=\sum_{s<t}\chi_s\gD L_s.
	\end{align*}
	Proposition~\ref{prop:upper} and Lemma~\ref{lem:spectral} imply
	$\chi_t=1$ eventually almost surely. On the other hand, the gated
	variance bound is deterministic:
	\begin{align*}
		\sum_{s\geq t_0}
		\E\left(\left(\frac{\chi_s\gD L_s}{s^{9/8}}\right)^2\right)
		\leq C\sum_{s\geq t_0}s^{-5/4}<\infty.
	\end{align*}
	Thus $\sum_s\chi_s\gD L_s/s^{9/8}$ converges almost surely and
	in $L^2$. Kronecker's lemma gives $\widetilde L_t=o(t^{9/8})$.
	The increments of $L_t$ and $\widetilde L_t$ differ at only finitely
	many times almost surely, so
	\begin{align*}
		L_t=o(t^{9/8})\quad\hbox{almost surely}.
	\end{align*}

\smallskip\noindent\textbf{Step 4: The first bootstrap, $\gl_t\gtrsim t^{1/8}$.}
	Since $Q_t\geq\sum_u d_u(t)=2(t-1)$,~\eqref{eq:Q-rho} first gives $\gl_t\gtrsim t^{1/8}$.
	Put $\cB_t=\sum_{s<t}\gl_s$. Then
	$\cB_t\gtrsim t^{9/8}$, and~\eqref{eq:Q-drift} implies
	$\sum_{s<t}b_s\geq2\ga\cB_t$.
	The martingale error from Step~3 is therefore negligible relative
	to this compensator. Hence
	\begin{align*}
		Q_t\sim\sum_{s<t}b_s,\qquad
		Q_t\geq\ga\cB_t\quad\hbox{eventually}.
	\end{align*}
	Combining the latter inequality with~\eqref{eq:Q-rho} gives
	\begin{align*}
		\gl_t\geq c\cB_t^{1/8}\quad\hbox{eventually},
	\end{align*}
	for a positive random constant $c$.

\smallskip\noindent\textbf{Step 5: The second bootstrap, $\gl_t\gtrsim t^{1/7}$.}
	The bound $\gD r_t\leq1$ gives $\gl_t=O(\sqrt t)$
	deterministically. Together with $\cB_t\gtrsim t^{9/8}$,
	this shows
	\begin{align*}
		\frac{\cB_{t+1}}{\cB_t}
		=1+\frac{\gl_t}{\cB_t}\to 1.
	\end{align*}
	For all sufficiently large $t$, the mean value theorem and Step~4
	now give
	\begin{align*}
		\cB_{t+1}^{7/8}-\cB_t^{7/8}
		&\geq\frac78\cB_{t+1}^{-1/8}\gl_t
		\geq\frac{7c}{8}
		\left(\frac{\cB_t}{\cB_{t+1}}\right)^{1/8}
		\geq c'>0.
	\end{align*}
	Summing gives $\cB_t\gtrsim t^{8/7}$, and then
	$\gl_t\gtrsim\cB_t^{1/8}\gtrsim t^{1/7}$.
	This proves the proposition and, in particular, $\gl_t=\Omega(t^{1/7})$. All constants in Steps~4--5 are
	pathwise comparison constants; the martingale estimates above
	were obtained with deterministic bounds and adapted gates.
\end{proof}

\section{Finite-core localization}\label{sec:core}

First, we prove existence of a finite old core.
\begin{prop}[A finite old core]\label{prop:finite-core}
	For every $0<\eps<1/16$, almost surely there is a finite $N$
	such that $d_u(t)\leq\eps\gl_t^2$ for every $t\geq N$ and
	$N<u\leq t$.
\end{prop}
\begin{proof}
	Put $\eta=\eps/12$ and
	$R=\max\{(2/\eta)^2,1/\eps\}$. For deterministic $N$, let
	$G_N=\{\gl^2_N\geq R\}$ and stop at the random time
	\begin{align*}
		\tau_N:=\inf\{t\geq N:\exists u>N, u\leq t, d_u(t)>\eps\gl_t^2\}.
	\end{align*}
	The connected old core $[N]$ has one edge to each exterior component.
	Root each component toward the core and let $h(u)\in[N]$ be the
	nearest core ancestor of $u$. This label is fixed at birth. For
	$N\leq s<\tau_N$ on $G_N$, leaf-to-root induction in the fixed tree
	$T_s$ gives, for each exterior child $w$,
	\begin{align*}
		\frac{v_w(s)}{v_{p(w)}(s)}
		=\frac1{\gl_s-\sum_{z:p(z)=w}v_z(s)/v_w(s)}
		\leq\frac1{(1-2\eps)\gl_s}\leq\frac2{\gl_s}.
	\end{align*}
	The base case is a leaf, and there are at most
	$\eps\gl_s^2$ children. Multiplication along the path to
	the core, using $2/\gl_s\leq\eta<1$, gives
	\begin{align}\label{eq:trap-comparison}
		a_s(u)\leq\eta a_s(h(u))+\gb/s.
	\end{align}
	Let $K=\gb(e-1)$. The two possible degree increments in~\eqref{eq:trap-comparison} are mutually exclusive, so
	\begin{align*}
		\E_s e^{\gD d_u(s)-3\eta\gD d_{h(u)}(s)}
		\leq1+a_s(h(u))\,\big(\eta(e-1)+e^{-3\eta}-1\big)+K/s
		\leq e^{K/s}.
	\end{align*}
	The expression in parentheses is negative because
	\begin{align*}
		\eta(e-1)+e^{-3\eta}-1
		\leq(e-4)\eta+\tfrac92\eta^2<0
		\text{ for }0<\eta<1/192.
	\end{align*}
	For a fixed label $u>N$, define
	$I_u=G_N\cap\{\tau_N>u\}\in\cF_u$. For $s\geq u$, define
	\begin{align*}
		L_s^{(u)}=\ind_{I_u}
		\exp\left\{d_{s\wedge\tau_N}(u)-3\eta d_{s\wedge\tau_N}(h(u))
		-K\sum_{j=u}^{s\wedge\tau_N-1}j^{-1}\right\}.
	\end{align*}
	On $I_u^c$ the process is defined as zero, so no unborn degree is
	evaluated. It is a nonnegative supermartingale starting at at most
	$e$, since $d_u(u)=1$. The random ancestor is $\cF_u$-measurable.

	On $G_N$ a newborn cannot cause the first violation, because
	$1\leq\eps\gl^2_N$. If $\tau_N=t<\infty$ and
	$\gl_t^2\geq t^{1/4}$, some existing $u<t$ therefore has
	\begin{align*}
		d_u(t)-3\eta d_{h(u)}(t)-K\sum_{j=u}^{t-1}j^{-1}
		>\tfrac34\eps t^{1/4}-K(1+\log t),
	\end{align*}
	where $d_{h(u)}(t)\leq\gl_t^2$. Every such offender satisfies $I_u$.
	Markov's inequality and a union bound give
	\begin{align*}
		\pr(G_N,\tau_N=t,\gl_t^2\geq t^{1/4})
		\leq e^{1+K}t^{K+1}e^{-3\eps t^{1/4}/4}.
	\end{align*}
	Consequently
	\begin{align*}
		\pr(\tau_N<\infty)\leq{} & \pr(G_N^c)
		+\pr(\exists t\geq N:\gl_t^2<t^{1/4})+\sum_{t\geq N}e^{1+K}t^{K+1}e^{-3\eps t^{1/4}/4}.
	\end{align*}
	The last series has vanishing tails. The first two terms tend to zero
	by Proposition~\ref{prop:bootstrap}, since $2/7>1/4$. Thus,
	$\pr(\tau_N<\infty)\to0$, and
	$
		\pr\left(\cap_N\{\tau_N<\infty\}\right)
		\leq\inf_N\pr(\tau_N<\infty)=0.
	$
	This proves the existence of a successful core.
\end{proof}
As a corollary, we get the following lower bound on the maximum Perron mass at a single site.
\begin{cor}[Positive core mass]\label{cor:coremass}
	Almost surely, $\liminf_t M_t>0$.
\end{cor}
\begin{proof}
	Choose a successful core $C=[N]$ with $\eps=1/64$. Write
	$A_t=\bigl(\begin{smallmatrix}A_C&B\\B^\top&A_F\end{smallmatrix}\bigr)$
	and $v=(z,w)$. The forest bound gives
	$\norm{A_F}\leq2\sqrt\eps\gl_t$. Since each exterior
	component attaches once to the core, $BB^\top$ is diagonal, with
	entries at most $D^\star_t\leq\gl_t^2$. Thus
	\begin{align*}
		\norm{w}\leq\frac{\norm{B}}{\gl_t-\norm{A_F}}\norm{z}
		\leq\frac{\norm{z}}{1-2\sqrt\eps}.
	\end{align*}
	Normalization implies $1= \norm{z}^2 +\norm{w}^2\le  (1+ (4/3)^2)\norm{z}^2=(5/3)^2\norm{z}^2$ and thus,  $M_t\geq\norm{z}/\sqrt N\geq3/(5\sqrt N)$.
\end{proof}
Finally, we prove the optimal growth rate for $D^\star_t, \gl_t,S_t$.
\begin{prop}[Matching scales]\label{prop:scales}
	Almost surely
	\begin{align*}
		D^\star_t=\Theta(t^{\theta}),\quad \gl_t,S_t=\Theta(t^{\theta/2}),\quad
		\liminf_t\gl_t\kappa_t>0.
	\end{align*}
\end{prop}
\begin{proof}
	First derive $S_t=O(\gl_t)$ dynamically. Fix a deterministic
	$m>0$ and a starting time $n$, and define the adapted stopping time
	$\tau=\inf\{t\geq n:M_t<m\}$. Put $U_t=\sqrt{Q_t}$, $z_t=\gD U_t$, and
	$x_t=\gD\gl_t^2$. Before this adapted stop,
	\begin{align*}
		0\leq z_t\leq2,\quad 0\leq x_t\leq1,\quad
		\E_t z_t\leq\frac{(2\ga+3\gb)\gl_t}{U_t},\quad
		\E_t x_t\geq\frac{2\ga m^3\gl_t}{U_t}.
	\end{align*}
	Here concavity of the square root gives
	$z_t\leq(d_{J_{t+1}}(t)+1)/U_t$, and
	$S_t\leq U_t/\gl_t$ gives
	$\kappa_t\geq m^3\gl_t/U_t$.
	Choose $a>0$ so small that
	$c_a(2\ga+3\gb)\leq\ga m^3/2$, where $c_a=(e^{2a}-1)/2$, and put
	$C=(\log2)/a$. Convex chord bounds yield
	\begin{align*}
		e^{az-aCx}\leq(1+c_a z)(1-x/2)\leq1+c_a z-x/2
		\quad(0\leq z\leq2,\ 0\leq x\leq1).
	\end{align*}
	In particular, before $\tau$, we have
	\begin{align*}
		\E_t e^{az_t-aCx_t}
		\leq 1+c_a\E_t z_t-\tfrac12\E_t x_t
		\leq 1-{\ga m^3\gl_t}/(2U_t)\leq1.
	\end{align*}
	No independence of the two increments is needed for this estimate.
	Thus, $(e^{a(U_{t\wedge\tau}-C\gl_{t\wedge\tau}^2)})_{t\geq n}$
	is a nonnegative supermartingale. Its initial value is integrable,
	since there are finitely many possible trees at the fixed time $n$.
	It converges and is therefore bounded almost surely. On the event that this stop never occurs,
	$U_t\leq C\gl_t^2+O(1)$, so $S_t\leq2C\gl_t$ eventually.
	The same bounds give
	$\gl_t\kappa_t\geq m^3\gl_t/S_t\geq m^3/(2C)>0$.
	Taking the countable union over $n$ and $m=1/k$ and using
	Corollary~\ref{cor:coremass} proves $S_t=O(\gl_t)$ and
	$\liminf_t\gl_t\kappa_t>0$.
	Finally, let $V_t:=\gl_t^3$. By~\eqref{eq:cubic-increment},
	\begin{align*}
		\E_t\gD V_t\geq3\gl_t(\ga\kappa_t+\gb/t).
	\end{align*}
	For its Doob martingale $L_t$, use the predictable gate
	\begin{align*}
		\chi_t=\ind_{\{\gl_t\leq5t^{\theta/2}\}},\quad
		\E_t(\chi_t\gD L_t)^2\leq Ct^{\theta}.
	\end{align*}
	The displayed variance bound follows from~\eqref{eq:cubic-increment}. Hence
	$\sum_t\E_t(\chi_t\gD L_t)^2/t^2
	\leq C\sum_t t^{-4/3}<\infty$, so
	$\sum_t\chi_t\gD L_t/t$ converges in $L^2$ and almost surely.
	As $\chi_t=1$ eventually by Proposition~\ref{prop:upper},
	$L_t=o(t)$. The positive pathwise lower bound on
	$\gl_t\kappa_t$ now gives $\gl_t^3\gtrsim t$.
	Combine with $D^\star_t\geq1+\gl_t^2/4$ and Proposition~\ref{prop:upper}.
	The spectral estimates give $\gl_t=\Theta(t^{\theta/2})$.
	If $h$ maximizes $v(t)$, its eigenvector equation gives
	$S_t\geq(1+\gl_t)M_t$. Core mass and the established upper bound
	give $S_t=\Theta(t^{\theta/2})$.
\end{proof}

\section{Competition and a persistent hub}\label{sec:competition}
In this section, we prove the uniqueness and persistence of the hub. For a vector with $r$ coordinates, $\cen(x)$ denotes the
vector obtained by subtracting the mean of its coordinates. Thus
\begin{align*}
	\cen(x)=x-\frac1r\sum_i x_i\cdot \vone.
\end{align*}
Here, $\vone$ is the all-ones vector of the relevant dimension. Matrix norms are
Euclidean operator norms and vector norms are Euclidean norms.

    \begin{thm}[One persistent hub]\label{thm:hub}
	Almost surely, there is a finite random vertex $h$ such that
	\begin{align*}
		d_h(t)=\gl_t^2+O(\gl_t),\quad
		\max_{u\ne h}d_u(t)=O(\gl_t),\quad
		v_h(t)\to2^{-1/2},\quad
		\max_{u\ne h}v_u(t)=O(\gl_t^{-1}).
	\end{align*}
	Thus $h$ is eventually the unique maximum-degree vertex.
\end{thm}

The proof has three stages. 
\begin{enumeratea}
    \item First, only finitely many vertices can retain degrees of order $\gl_t^2$. The finite active-set Lemma~\ref{lem:active} identifies the possible survivors. 
    \item Next, the competition-window Lemma~\ref{lem:window} shows that surviving candidates must have degrees within $O(\gl_t)$ of one another. 
    \item Finally, a squared imbalance statistic for their effective Perron matrix has positive drift even at perfect balance. A bounded stopped version of this statistic rules out two or more survivors.
\end{enumeratea}
We first strengthen the exterior bound in Proposition~\ref{prop:finite-core}.
Choose a successful connected old core $C=[N]$, with exterior degrees
at most $\eps\gl_t^2$, where $\eps<1/16$.
The directed tree recursion
$$
r_{u\to p}=(\gl_t-\sum_{w\text{ child of }u}r_{w\to u})^{-1}
$$
and induction from the leaves give $r_{u\to p}\leq2/\gl_t$.
Every exterior vertex therefore has
$v_u(t)\leq2\max_{i\in C}v_i(t)/\gl_t$. By Proposition~\ref{prop:scales},
\begin{align*}
	a_t(u)\leq ct^{-\theta}+\gb/t\text{ for }u\notin C.
\end{align*}
On a fixed stopped localization region, the number $X_t(u)$ of
post-start attachments to any specified vertex satisfies
$\E e^{X_t(u)}\leq\exp(c't^{\theta/2})$. This follows by iterating
$\E(e^{\gD X_s}\mid\cF_s)\leq
	\exp((e-1)(cs^{-\theta}+\gb/s))$. Unborn vertices have zero indicators.
For sufficiently large $A$, the probability of
$d_u(t)>At^{\theta/2}$ is at most $C_{\mathrm{loc}}e^{-ct^{\theta/2}}$.
Here $C_{\mathrm{loc}}$ and $c>0$ depend only on the fixed
localization bounds and starting time. The union bound over $u\leq t$
is summable over $t$.
Borel--Cantelli lemma and the countable localization convention yield
\begin{align}\label{eq:exterior}
	\max_{u\notin C}d_u(t)=O(\gl_t)\quad\text{almost surely}.
\end{align}

\begin{lem}[Finite Schur estimates]\label{lem:schur}
	Fix a nonempty finite set $I$ of $r$ vertices, all already born. Define $F=V(T_t)\setminus I$.
	Write the adjacency blocks as $A_I,B,A_F$. Suppose
	$\max_{u\in F}d_u(t)\leq C\gl_t$ and $t\leq C\gl_t^3$. For
	$x\in[\gl_t/2,2\gl_t]$, define
	\begin{align*}
		R(x)=(x\mathrm{Id}-A_F)^{-1},\quad
		\cH(x)=A_I+BR(x)B^\top,\quad \cK(x)=x\cH(x).
	\end{align*}
	For all sufficiently large $t$, uniformly on this interval,
	\begin{align}
		\cH(\gl_t)v_I & =\gl_tv_I,                                                        &
		\cK(x)            & =\diag(d_i(t):i\in I)+O(\gl_t),\label{eq:schur-size}   \\
		\cK'(x)           & =O(1),                                                                &
		\cH'(x)           & =-\diag(d_i(t)/x^2:i\in I)+O(\gl_t^{-1}).
		\label{eq:schur-derivative}
	\end{align}
	The off-diagonal part of $\cH(x)$ is nonnegative and bounded, and
	$\norm{BR(x)e_u}=O(\gl_t^{-1})$ for every $u\in F$.
\end{lem}
\begin{proof}
	Applying the forest bound of Lemma~\ref{lem:spectral} componentwise,
	with norm zero for an edgeless component, gives
	$\norm{A_F}\leq2\sqrt{\max_{u\in F}d_u(t)}=O(\sqrt{\gl_t})$. Eliminating the exterior Perron
	equation gives the first identity, while the convergent expansion is
	\begin{align*}
		\cK(x)=xA_I+BB^\top+\sum_{k\geq1}x^{-k}BA_F^kB^\top.
	\end{align*}
	Both $BB^\top-\diag(d_i(t):i\in I)$ and $BA_FB^\top$ are
	bounded. The diagonal discrepancy in the first is the bounded
	internal degree, and its off-diagonal entries count length-two paths.
	The second counts length-three paths through $F$. Each such path is
	unique in a tree. Each entry of $A_FB^\top$ is at most one, so
	$\norm{BA_F}^2\leq\tr(BA_F^2B^\top)\leq rt$.
	Consequently
	$\norm{BA_F^kB^\top}\leq rt\norm{A_F}^{k-2}$ for $k\geq2$.
	Since $x\asymp\gl_t$ and $\norm{A_F}/x=O(\gl_t^{-1/2})$,
	the terms with $k\geq2$ have total norm $O(t/x^2)=O(\gl_t)$.
	Their derivatives in $\cK$ have total norm $O(t/x^3)=O(1)$.
	For $\cH=\cK/x$, the corresponding derivative remainder is
	$O(t/x^4)=O(\gl_t^{-1})$. The terms $xA_I$, $BB^\top$ and
	$x^{-1}BA_FB^\top$ satisfy the asserted bounds directly. This proves~\eqref{eq:schur-size}--\eqref{eq:schur-derivative}.

	The same cavity induction, now with exterior maximum degree
	$O(\gl_t)=o(x^2)$, gives
	\begin{align*}
		0\leq R(x)_{ab}\leq(2/x)^{\dist_F(a,b)+1}
	\end{align*}
	within each component, and zero between components. The diagonal Green
	entry is at most $2/x$, and the unique path contributes one cavity
	factor, at most $2/x$, per edge. A component of $F$ contains at most
	one neighbor of any specified $i\in I$, since otherwise there would
	be a cycle through $i$. Hence each coordinate of $BR(x)e_u$ is at
	most $2/x$, proving the final assertion. For distinct $i,j\in I$,
	at most one exterior component can connect a neighbor of $i$ to a
	neighbor of $j$, and both endpoint neighbors in that component are
	unique; otherwise the tree would contain a cycle. The Green bound
	therefore gives the explicit estimate
	$
		0\leq\cH_{ij}(x)\leq(A_I)_{ij}+2/x.
	$
\end{proof}

The next Lyapunov argument follows the stochastic-approximation
viewpoint discussed in~\cite{Benaim1999}. We give its convergence
steps directly, since the effective Schur matrix depends on time.

\begin{lem}[A finite active set]\label{lem:active}
	Almost surely, some nonempty finite random set $I$ satisfies
	\begin{align*}
		\max_{i\in I}\abs{d_i(t)/\gl_t^2-1}\to 0,\quad
		\max_{u\notin I}d_u(t)=O(\gl_t),\quad
		\sum_{i\in I}v_i(t)\geq c>0
	\end{align*}
	eventually. Moreover $\gl_t^2=D^\star_t+O(\gl_t)$.
\end{lem}
\begin{proof}
	Apply Lemma~\ref{lem:schur} to $C$ in~\eqref{eq:exterior}.
	The matrix $\cK(\gl_t)$ is nonnegative, since its resolvent has a
	convergent nonnegative Neumann expansion. Its eigenvector $v_C$ is
	strictly positive. Hence, its eigenvalue $\gl_t^2$ is the spectral
	radius. Therefore
	$\gl_t^2=\max_{i\in C}d_i(t)+O(\gl_t)=D^\star_t+O(\gl_t)$.
	The core mass bound of Corollary~\ref{cor:coremass} gives
	$\norm{v_C(t)}\geq c>0$. 

    Define
	$D_t=\sum_{i\in C}d_i(t)$, $x_i=d_i(t)/D_t$, $\cV_t=\sum_{i\in C}x_i^2$,
	and $P_t=\sum_{i\in C}a_t(i)=\Theta(\gl_t^{-1})$.
	Here $D_t=\Theta(\gl_t^2)$. Summing the Schur equation gives
	\begin{align*}
		\frac{\sum_{i\in C}d_i(t)v_i(t)}{\sum_{i\in C}v_i(t)}
		\geq\max_{i\in C}d_i(t)-C\gl_t.
	\end{align*}
	The conditional parent law $p$ on a core hit is the Perron law mixed
	with the uniform law, whose weight is $O(\gl_t/t)$.
	Thus $\sum_i x_ip_i=\max_i x_i+\epsilon_t$, with
	$-C/\gl_t\leq\epsilon_t\leq0$.
	Writing 
    \[
    \phi(x)=\max_i x_i-\sum_i x_i^2\geq0,
    \]
    a core hit at $i$
	updates $x$ to $(D_tx+e_i)/(D_t+1)$ and gives exactly
	\begin{align*}
		\gD\cV_t
		=\frac{2D_t(x_i-\cV_t)+(1-\cV_t)}{(D_t+1)^2}.
	\end{align*}
	Outside a core hit the statistic is unchanged. Averaging this identity
	gives
	\begin{align*}
		\E(\gD \cV_t\mid\cF_t)
		=\gc_t\,(\phi(x(t))+\epsilon_t)+O(P_t/D_t^2),
		\quad\gc_t=\frac{2P_t}{D_t+1}=\Theta(t^{-1}).
	\end{align*}
	The negative error is 
    $O(t^{-(1+\theta/2)})$, and $P_t/D_t^2=O(t^{-(1+\theta)})$.
	On a region stopped at $\tau$, the corrected process
	\begin{align*}
		\cV_{t\wedge\tau}
		+C_1\sum_{s=n}^{t-1}\ind_{\{s<\tau\}}s^{-(1+\theta/2)}
	\end{align*}
	is a bounded submartingale with drift at least
	$\ind_{\{t<\tau\}}\gc_t\phi(x(t))$.
	Its convergence and bounded expectation imply, on $\{\tau=\infty\}$,
	$\cV_t$ converges and $\sum_t\gc_t\phi(x(t))<\infty$.

	Each coordinate of $x$ has drift $O(t^{-1})$ and martingale
	variance bounded by $\sum_tO(t^{-(1+\theta)})<\infty$.
	Its martingale tails therefore tend uniformly to zero, while its
	drift on $[t,(1+\eta)t]$ is $O(\log(1+\eta))=O(\eta)$.
	Suppose that $\phi(x(t))\geq\delta>0$ at arbitrarily late times.
	The function $\phi$ is uniformly continuous on the compact simplex.
	Choose $\varepsilon>0$ so that simplex points at distance at most
	$\varepsilon$ have $\phi$ values differing by at most $\delta/2$.
	Then choose a fixed $\eta>0$ so small that the drift displacement over
	$[t,(1+\eta)t]$ is at most $\varepsilon/2$.
	Uniformly vanishing martingale tails bound their additional displacement
	by $\varepsilon/2$ for all sufficiently large $t$.
	Thus $\phi(x(s))\geq\delta/2$
	throughout every such interval starting at a sufficiently late time
	with $\phi(x(t))\geq\delta$.
    Extract infinitely many disjoint intervals
	of this form. Since $\gc_s\asymp1/s$, each contributes a fixed
	positive amount to $\sum_s\gc_s\phi(x(s))$, a contradiction.
	Therefore $\phi(x(t))\to0$. 
    
    To identify its zero set, let $m=\max_i x_i$. Then
	$\phi(x)=\sum_i x_i(m-x_i)$, so $\phi(x)=0$ exactly when every
	positive coordinate equals $m$. These are the uniform laws on nonempty
	subsets of $C$. This is a finite set. Since $\phi(x(t))\to0$,
	$x(t)$ approaches this set. Also $\gD x(t)\to0$, so it cannot pass
	infinitely often between disjoint sufficiently small neighborhoods of
	different members without leaving their union. It therefore converges
	to one member, uniform on a nonempty subset $I$. The relation $$\gl_t^2=\max_{i\in C}d_i(t)+O(\gl_t)$$ now gives
	$d_i(t)/\gl_t^2\to1$ for $i\in I$, and $d_j(t)/\gl_t^2\to0$ otherwise.
    
	It remains to improve the degree bound for inactive core vertices.
	First, fix a deterministic core $C=[N]$, a label $j\in C$, a starting
	time, and rational core and scale bounds. Stop at the first failure of
	$d_j(t)\leq\gl_t^2/2$ or of these bounds. Before this stop, its Schur equation has
	diagonal gap $\gl_t-\cH_{jj}(\gl_t)\geq\gl_t/3$ for all
	sufficiently large $t$. The bounded off-diagonal entries therefore give
	$v_j\leq C_1/\gl_t$, hence $a_t(j)\leq C_2t^{-\theta}$.
	The stopped exponential concentration argument used in~\eqref{eq:exterior} gives $d_j(t)=O(\gl_t)$ on every event where
	these bounds never fail. For each inactive label the degree-ratio limit
	ensures that some deterministic starting time and bounds succeed.
	Taking the countable union over $N$, labels, starting times and bounds
	proves the required bound for all inactive core vertices.
	The core mass bound then gives the asserted positive active mass.
\end{proof}

The active-set lemma permits several survivors with equal leading
degrees. We now narrow the possible ties. A degree imbalance much larger
than $\gl_t$ has outward drift, and a return to the smaller window
is exponentially unlikely. Such an excursion cannot persist on an
active-set event, where all degree differences are $o(\gl_t^2)$.

\begin{lem}[The competition window]\label{lem:window}
	For a deterministic finite nonempty $I$, define
	$L_t^I=\sum_{i\in I}v_i(t)$,
	$U_t^I=\norm{\cen(d_i(t):i\in I)}$, and
	\begin{align}\label{eq:active-events}
		\cE_I=
		\left\{{d_i(t)}/{\gl_t^2}\to1\ \text{ for }i\in I,\quad
		\max_{u\notin I}d_u(t)=O(\gl_t),\quad
		\liminf_t L_t^I>0\right\}.
	\end{align}
	Simultaneously, for all such deterministic $I$, almost surely on
	$\cE_I$ we have $U_t^I=O(\gl_t)$.
	Consequently, the bound holds for the finite random active set in
	Lemma~\ref{lem:active}.
\end{lem}
\begin{proof}
	The events $\cE_I$ belong to
	$\cF_\infty=\gs(\cup_t\cF_t)$. Asymptotic $O$ bounds are countable
	unions of eventual bounds with rational constants. Lemma~\ref{lem:active}
	gives $\pr(\cup_I\cE_I)=1$. We will use these events by
	intersection only after proving the stopped estimates.
	Only $r=|I|\geq2$ needs consideration. Abbreviate $U_t=U_t^I$.

	\smallskip\noindent\emph{Local bounds and outward drift.}
	Fix a deterministic set $I$,
	a deterministic $n\geq\max\{t_0,\max I\}$, and rational constants
	$0<c<C$. Define the adapted gates
	\begin{align}\label{eq:competition-gates}
		\cG_t=
		\left\{ct^{\theta/2}\leq\gl_t,S_t\leq Ct^{\theta/2},\quad
		\max_{u\notin I}d_u(t)\leq C\gl_t,\quad L_t^I\geq c\right\}.
	\end{align}
	Let $\tau=\inf\{t\geq n:\cG_t\text{ fails}\}$, with
	$\tau=n$ if the initial gates fail. All estimates that follow are
	proved on $\{t<\tau\}$ for this fixed $I$. No limiting property of
	the random active set is used to obtain them.

	Write $P_t=\sum_{i\in I}a_t(i)$ and let $p_t$ be the conditional parent
	law on an $I$ hit. The gates give deterministic two-sided bounds
	$P_t\asymp\gl_t^{-1}\asymp t^{-\theta/2}$. The summed Schur equation
	gives a Perron-weighted mean degree at least $\max_{i\in I}d_i(t)-C_1\gl_t$.
	The uniform component of $p_t$ has weight
	$r\gb/(tP_t)=O(\gl_t/t)$ and changes this mean by at most $O(1)$,
	since $\max_{i\in I}d_i(t)\leq\gl_t^2=O(t^{\theta})$. Thus
	\[
    \sum_{i\in I}d_i(t)p_t(i)\geq\max_{i\in I}d_i(t)-C_2\gl_t.
    \]
	Writing $\bar d=r^{-1}\sum_{i\in I}d_i(t)$, convexity of the norm implies,
	for $U_t>0$,
	\begin{align*}
		\E(\gD U_t\mid\cF_t)
		\geq\frac{P_t}{U_t}(\max_{i\in I}d_i(t)-\bar d-C_2\gl_t).
	\end{align*}
	A zero-sum $r$-vector has maximum at least its norm divided by
	$\sqrt{r(r-1)}$. For a sufficiently large deterministic $K$,
	$b(t)=Kt^{\theta/2}$ therefore satisfies
	\begin{align*}
		U_t\geq b(t)\quad\longrightarrow \quad
		\E_t\gD U_t\geq a_0t^{-\theta/2}.
	\end{align*}
	Also $|\gD U_t|\leq1$, and it is zero outside an $I$ hit, so
	$\E_t(\gD U_t)^2\leq P_t\leq A_0t^{-\theta/2}$ throughout the gates.
	For all sufficiently large $t$, $V_t=U_t-b(t)$ consequently satisfies,
	when $U_t\geq b(t)$ and $t<\tau$,
	\begin{align*}
		|\gD V_t|\leq2,\quad
		\E_t\gD V_t\geq\tfrac12a_0t^{-\theta/2},\quad
		\E_t(\gD V_t)^2\leq A_1t^{-\theta/2}.
	\end{align*}
	The inequality
	$e^{-\zeta z}\leq1-\zeta z+\tfrac12\zeta^2e^{2\zeta}z^2$
	for $|z|\leq2$ makes $e^{-\zeta V_t}$ a supermartingale before exit
	or a return below $b$, for a fixed sufficiently small $\zeta>0$.

	\smallskip\noindent\emph{A large excursion rarely returns.}
	For every deterministic sufficiently late $s$, define
    \[
    \gs_s:=\inf\{t\geq s:U_t<b(t)\}\text{ and
	}
    A_s:=\{s<\tau,\ U_s\geq2b(s)\}\in\cF_s.
    \]
	To define the exponential on the whole probability space, let
	$T_s=s$ on $A_s^c$, and $T_s=\gs_s\wedge\tau$ on $A_s$.
	Since $A_s\in\cF_s$, $T_s$ is a stopping time with $T_s\geq s$.
	For $t\geq s$, define
	\begin{align*}
		Z_t^{(s)}=\ind_{A_s}
		\exp(-\zeta V_{t\wedge T_s}).
	\end{align*}
	On $\{A_s,t<T_s\}$ the pre-step gates and $U_t\geq b(t)$ hold.
	The one-step exponential estimate applies also to the transition
	attaining $T_s$, since $|\gD U_t|\leq1$ for every possible parent.
	Thus $Z^{(s)}$ is a nonnegative supermartingale. It is integrable at
	every finite horizon (indeed $V_{t\wedge T_s}\geq-b(t)$).
	For every deterministic integer $m\geq s$,
	\begin{align*}
		\E(Z_m^{(s)}\mid\cF_s)
		\leq\ind_{A_s}e^{-\zeta V_s}.
	\end{align*}
	On $A_s\cap\{\gs_s<\tau\}$ the terminal exponential is at least
	one. Conditional Fatou's lemma therefore gives
	\begin{align*}
		\ind_{A_s}\pr(\gs_s<\tau\mid\cF_s)
		\leq\ind_{A_s}e^{-\zeta V_s}.
	\end{align*}
	In particular, on $A_s$,
	\begin{align*}
		\pr(\gs_s<\tau\mid\cF_s)
		\leq e^{-\zeta(U_s-b(s))}\leq e^{-\zeta b(s)}.
	\end{align*}
	The argument requires only nonnegativity.
	In particular
	\begin{align}\label{eq:window-return}
		\sum_s\pr(A_s\cap\{\gs_s<\tau\})<\infty.
	\end{align}
	This sum runs over all deterministic integer starting times, so it
	also controls every possible random excursion time.

	\smallskip\noindent\emph{An active set cannot escape the window forever.}
	The stopped martingale differences
	\[
    \xi_{t+1}=\ind_{\{t<\tau\}}(\gD U_t-\E_t\gD U_t)
    \]
	satisfy $\E_t\xi_{t+1}^2\leq A_0t^{-\theta/2}$.
	Thus $\sum_t\xi_{t+1}/t^{\theta}$ converges in $L^2$ and almost surely,
	and $\sum_{s<t}\xi_{s+1}=o(t^{\theta})$.
	On $\{\tau=\infty,\ \gs_s=\infty\}\cap A_s$, the positive drift
	therefore gives
	$U_t\geq\tfrac12a_0\sum_{j=s}^{t-1}j^{-\theta/2}+o(t^{\theta})$ and
	$\liminf_t U_t/t^{\theta}>0$.

	The martingale convergence and return conclusions hold on one
	probability-one event for every integer $s$. Intersect these events
	also over all deterministic finite $I$, integer $n$ and rational
	$0<c<C$. This is a countable intersection.
	Only now intersect with $\cE_I$. Its degree limits imply
    \[
    U_t=o(\gl_t^2)=o(t^{\theta})
    \]
    on the probability-one scale event of
	Proposition~\ref{prop:scales}.
	If also $\tau=\infty$, every late $A_s$ must therefore have
	$\gs_s<\infty=\tau$, outside a single null event. By
	\eqref{eq:window-return} and Borel--Cantelli lemma there are only finitely
	many such $A_s$. Consequently $U_t<2b(t)$ eventually.
	Every path in $\cE_I$ and the scale event eventually satisfies
	some gates of~\eqref{eq:competition-gates}. The countable union over
	$n,c,C$ proves the assertion on $\cE_I$, simultaneously for all
	deterministic $I$. Taking $\cup_I\cE_I$ then gives the random
	active-set conclusion.
\end{proof}

To exclude ties inside this window, we use the row sums of the exact
Schur matrix. Their centered vector measures the effective imbalance
between candidates. The next identity shows that Perron selection
reinforces this imbalance. Its coefficient remains positive when some
effective couplings tend to zero.

\begin{lem}[Perron repulsion]\label{lem:repulsion}
	Let $H=\diag(y)+A$ be a symmetric $r\times r$ matrix, with $A$ nonnegative
	and zero diagonal. Let $\psi>0$ be a Perron vector of $H$. Define
	$p=\psi/\sum_i\psi_i$, and $Y=\cen(H\vone)$. Then
	\begin{align}\label{eq:sym}
		Y\cdot(p-r^{-1}\vone)
		=\frac1r\sum_{i<j}A_{ij}
		(\psi_j/\psi_i+\psi_i/\psi_j-2)
        =\frac1r\sum_{i<j}A_{ij}
		\left(\sqrt{\psi_j/\psi_i}-\sqrt{\psi_i/\psi_j}\right)^2.
	\end{align}
	On bounded sets of $A$ and $Y$, this is at least $c\norm{Y}^2$ for
	a positive constant that does not require a lower bound on nonzero $A_{ij}$.
\end{lem}
\begin{proof}
	Let $\ell$ be the Perron eigenvalue of $H$, and define 
	$b_i=\sum_jA_{ij}(1-\psi_j/\psi_i)=(H\vone)_i -\ell$. In particular, we have 
	$Y=\cen(b)$. Symmetry gives $b\cdot p=0$ and
	$-\sum_i b_i=\cE$, the edge sum displayed in~\eqref{eq:sym}.
	This proves the identity. 
    
    Subtracting the mean of $y$ from $H$,
	bounded $A,Y$ imply bounded $y,\ell$; hence, bounded
	$s_i=\sum_jA_{ij}\psi_j/\psi_i=\ell-y_i$.
	For $a_{ij}=\psi_j/\psi_i$, write
	$b_i=\sum_j\sqrt{A_{ij}a_{ij}}\,
		\sqrt{A_{ij}}(1-a_{ij})/\sqrt{a_{ij}}$.
	Cauchy--Schwarz inequality then gives
	\begin{align*}
		b_i^2\leq s_i\sum_jA_{ij}(\psi_j/\psi_i+\psi_i/\psi_j-2).
	\end{align*}
	More explicitly, let $r_i=(H\vone)_i$, so
	$y_i=r_i-(A\vone)_i$. The Perron eigenvalue is at most
	$\max_i r_i$. Therefore
	$s_i=\ell-r_i+(A\vone)_i\leq\max_jr_j-r_i+(A\vone)_i$, which gives
	$s_i\leq2\norm{Y}_\infty+\norm{A\vone}_\infty$.
	With $B_0=1+2\norm{Y}_\infty+\norm{A\vone}_\infty$,
	\begin{align*}
		\norm{Y}^2\leq\norm{b}^2\leq2B_0\cE,\quad
		Y\cdot(p-r^{-1}\vone)\geq\frac{\norm{Y}^2}{2rB_0}.
	\end{align*}
	This gives a uniform positive coefficient on bounded sets even when
	some or all off-diagonal entries tend to zero.
\end{proof}

\begin{lem}[Bounded stopped drift]\label{lem:bounded-stopped-drift}
	Let $\tau\geq n$ be a stopping time and $(Z_t)_{t\geq n}$ an
	adapted process satisfying $0\leq Z_t\leq B$ on the whole probability
	space for a deterministic $B<\infty$. If deterministic numbers
	$c_t\geq0$ satisfy
	\begin{align*}
		\E_t\gD Z_t\geq c_t\ind_{\{t<\tau\}},\qquad
		\sum_{t=n}^{\infty}c_t=\infty,
	\end{align*}
	then $\pr(\tau=\infty)=0$.
\end{lem}
\begin{proof}
	For every deterministic $T>n$, taking expectations and summing gives
	\begin{align*}
		B\geq\E Z_T\geq\E Z_n+
		\sum_{t=n}^{T-1}c_t\pr(t<\tau)
		\geq\E Z_n+\pr(\tau=\infty)\sum_{t=n}^{T-1}c_t.
	\end{align*}
	Letting $T\to\infty$ proves the claim. In an application where $Z$
	is frozen at $\tau$, its deterministic upper bound must include the
	value at the exit step.
\end{proof}

\subsection{Proof of Theorem~\ref{thm:hub}}
	Lemma~\ref{lem:active} supplies a finite active set. We must show that it
	has only one member. For a nonempty candidate set $I$, the complement adjacency matrix
	$A_F$ is a proper principal submatrix of the connected tree adjacency
	matrix. Hence $\rho(A_F)<\gl_t$, so the resolvent and the Schur
	matrix are defined at $x=\gl_t$ for every such candidate set.
	Use the exact Schur matrix
	$\cH_t(x)$ from Lemma~\ref{lem:schur} and its centered row-sum vector
	\begin{align*}
		Y_t=\cen(\cH_t(\gl_t)\vone).
	\end{align*}
	Under the exterior and scale bounds used below, the size estimate gives
	\begin{align}\label{eq:imbalance-size}
		Y_t=\gl_t^{-1}\cen(d_i(t):i\in I)+O(1).
	\end{align}
	Thus Lemma~\ref{lem:window} makes $Y_t$ bounded whenever $I$ is an active
	set. The Schur correction in $Y_t$ records the interactions between
	candidates through the rest of the tree. Keeping this correction is
	important because the $O(1)$ error in~\eqref{eq:imbalance-size} is on the same
	scale as the bounded imbalance itself, whereas the Perron-repulsion
	inequality applies to the exact Schur matrix.

	Perron reinforcement pushes a nonzero imbalance outward. Even at
	$Y_t=0$, candidate attachments supply positive squared increments.
	Together these give $\norm{Y_t}^2$ drift of order at least $1/t$.
	This exclusion of coexistence is related to the avoidance of unstable
	limits in reinforced systems~\cite{Pemantle1990}; here the required
	estimates are proved directly. We stop at the first failure of fixed
	deterministic bounds, control the exit value, and apply
	Lemma~\ref{lem:bounded-stopped-drift}. A countable cover then accounts
	for every possible coexistence path.

	\smallskip\noindent\textbf{Step 1. Fix the candidates and the stopping region.}
	Fix a deterministic finite $I$, $r=|I|\geq2$, rational bounds
	$0<c<C$ and $K,L>0$, and a deterministic integer
	$n\geq n_0(I,c,C,K,L)\geq\max\{t_0,\max I\}$.
	The threshold $n_0$ is chosen large enough for the estimates below.
	Write $u_0=r^{-1}\vone$ and $U_t=U_t^I$.
	The gates $\cG_t$ in~\eqref{eq:competition-gates} bound the growth
	scales, the exterior degrees, and the total candidate centrality. Add
	bounds on the degree spread and the row-sum imbalance by setting
	\begin{align*}
		\cD_t=\cG_t\cap
		\{U_t\leq K\gl_t,\ \norm{Y_t}\leq L\},\quad
		\tau=\inf\{t\geq n:\cD_t\text{ fails}\},\quad
		A_n=\cD_n\in\cF_n.
	\end{align*}
	If the initial gates fail, $\tau=n$.
	All calculations through Step 4 use these deterministic choices and
	the adapted pre-step event $\{t<\tau\}\in\cF_t$, so conditional drifts
	use the original attachment law. The future active-set event $\cE_I$
	enters only in the countable cover of Step 5.

	Before $\tau$, the derivative estimate~\eqref{eq:schur-derivative} gives
	\begin{align}\label{eq:centered-derivative}
		\frac d{dx}\cen(\cH_t(x)\vone)=O(\gl_t^{-1})
	\end{align}
	uniformly for $x\in[\gl_t/2,2\gl_t]$.
	Indeed, centering the leading diagonal term leaves the degree-spread
	bound $U_t\leq K\gl_t$ divided by $x^2$.

	\smallskip\noindent\textbf{Step 2. Control one attachment, including the exit step.}
	We need bounds that hold for every possible parent, even when the
	updated tree fails a gate. First inherit rough geometric bounds from
	the current state. Write $\gl=\gl_t$ and
	$\gl'=\gl_{t+1}$. The pendant-edge bound gives
	\begin{align}\label{eq:exit-spectral}
		\gl\leq\gl'\leq\sqrt{\gl^2+1}
		\leq\gl+1/(2\gl).
	\end{align}
	The new vertex is outside $I$, so
	\begin{align}\label{eq:exit-geometry}
		\begin{aligned}
			\max_{u\notin I}d_u(t+1) & \leq C\gl+1\leq(C+1)\gl', \\
			t+1                         \leq(c^{-3}+1)(\gl')^3, &\text{ and }
			U_{t+1}                     \leq U_t+1\leq(K+1)\gl'.
		\end{aligned}
	\end{align}
	Here $\gl\geq1$. These enlarged deterministic bounds allow
	Lemma~\ref{lem:schur} to be applied to the updated tree even if its
	$S$, candidate-mass, or imbalance bound fails. In particular,
	$\norm{A_{F,t+1}}=O(\sqrt\gl)<\gl$ and
	$[\gl,\gl']\subseteq[\gl'/2,2\gl']$ for large $t$.
	The updated Schur matrix is therefore defined at the old radius, and
	its centered derivative is $O(\gl^{-1})$ on this interval.

	Now calculate the matrix update at the fixed radius $x=\gl_t$.
	Attaching a leaf at $i\in I$ adds $e_ie_i^\top/x$ to $\cH_t(x)$.
	For an attachment at $u\in F$, eliminating the new leaf changes
	the inverse exterior resolvent from $R(x)^{-1}$ to
	$R(x)^{-1}-e_ue_u^\top/x$. The Sherman--Morrison identity therefore
	adds to $\cH_t(x)$ the matrix
	\begin{align*}
		E_u=\frac{(BR(x)e_u)(BR(x)e_u)^\top}{x-R(x)_{uu}}
		=O(\gl_t^{-3}).
	\end{align*}
	The derivative $\cH'(x)=-BR(x)^2B^\top$ is negative semidefinite.
	Hence the top eigenvalue of $\cH(x)$ is nonincreasing in $x$.
	For clarity, let $\cH_{\rm new}$ be the updated Schur matrix and define $f(x)=\lambda_{\max}(\cH_{\rm new}(x))-x$.
	Since $\cH_{\rm new}^{\prime}(x)$ is negative semidefinite,
	$f(y)\leq f(x)-(y-x)$ for $y\geq x$ in the interval under consideration.
	At the old radius, Weyl's inequality gives
	$0\leq f(\gl_t)\leq\norm{\cH_{\rm new}(\gl_t)-\cH_t(\gl_t)}$.
	Since $f(\gl_{t+1})=0$, it follows that
	\begin{align*}
		0\leq\gD\gl_t
		\leq\norm{\cH_{\rm new}(\gl_t)-\cH_t(\gl_t)}.
	\end{align*}
	Thus $\gD\gl_t\leq\gl_t^{-1}$ on an $I$ hit and
	$\gD\gl_t=O(\gl_t^{-3})$ on an exterior hit.
	This sharper exterior estimate uses the updated Schur matrix only
	after its existence has been established by the rough bounds above.
	Accounting for the change from $\gl$ to $\gl'$ with
	\eqref{eq:centered-derivative} now gives
	\begin{align}\label{eq:Y-updates}
		\gD Y_t=
		\begin{cases}
			(e_i-u_0)/\gl_t+O(\gl_t^{-2}), & i\in I, \\
			O(\gl_t^{-3}),                     & u\in F.
		\end{cases}
	\end{align}
	The estimates are uniform over all possible parents and include the
	transition that causes $\tau$.

	\smallskip\noindent\textbf{Step 3. Separate reinforcement from attachment noise.}
	Put $z_t=v_I/L_t^I$ and $P_t=\sum_{i\in I}a_t(i)\asymp\gl_t^{-1}$.
	Given an $I$ hit, the parent law is exactly
	\begin{align*}
		p_t=(1-\epsilon_t)z_t+\epsilon_tu_0,\quad
		\epsilon_t=\frac{r\gb}{tP_t}=O(\gl_t^{-2}).
	\end{align*}
	The mean of the errors in~\eqref{eq:Y-updates} is
	$O(P_t\gl_t^{-2}+\gl_t^{-3})=O(\gl_t^{-3})$.
	Consequently
	\begin{align}\label{eq:Y-mean}
		\E_t\gD Y_t
		=\frac{P_t(1-\epsilon_t)}{\gl_t}(z_t-u_0)
		+O(\gl_t^{-3}).
	\end{align}
	Lemma~\ref{lem:repulsion} applies to $\cH_t(\gl_t)$, whose off-diagonal
	part is bounded on the gates. It gives
	$Y_t\cdot(z_t-u_0)\geq a_1\norm{Y_t}^2$ with a deterministic $a_1>0$.
	This is the outward contribution from reinforcement.

	There is also a positive contribution from the squared increments.
	Every candidate hit has leading increment $(e_i-u_0)/\gl_t$, with
	$\norm{e_i-u_0}^2=1-1/r>0$. Discarding the nonnegative exterior terms,
	\begin{align}\label{eq:Y-noise}
		\E_t\norm{\gD Y_t}^2
		\geq\frac{P_t(1-1/r)}{\gl_t^2}-C_3\gl_t^{-4}
		\geq c_0\gl_t^{-3}-C_3\gl_t^{-4}
	\end{align}
	for a deterministic $c_0>0$. Candidate hits have probability of order
	$\gl_t^{-1}$ and increments of order $\gl_t^{-1}$, explaining the
	order $\gl_t^{-3}$. This is a conditional second-moment estimate,
	rather than a conditional variance estimate. At $Y_t=0$ the mean
	vector increment may cancel, but the expected squared increment
	remains positive. The factor $1-1/r$ is positive precisely when there
	is more than one candidate.

	Using the exact identity
	$\gD\norm{Y_t}^2=2Y_t\cdot\gD Y_t+\norm{\gD Y_t}^2$ gives
	\begin{align*}
		\E_t\gD\norm{Y_t}^2
		\geq\frac{a\norm{Y_t}^2}{\gl_t^2}
		+\frac{c_0}{\gl_t^3}
		-\frac{C_4\norm{Y_t}}{\gl_t^3}
		-\frac{C_4}{\gl_t^4}
	\end{align*}
	with a deterministic $a>0$. Complete the square using
	\begin{align*}
		\frac{az^2}{\gl^2}-\frac{C_4z}{\gl^3}
		\geq-\frac{C_4^2}{4a\gl^4}.
	\end{align*}
	All losses are now $O(\gl_t^{-4})$, smaller than the positive
	noise term. After increasing $n_0$ if necessary,
	\begin{align}\label{eq:Y-positive-drift}
		\E_t\gD\norm{Y_t}^2
		\geq\frac{c_0}{2\gl_t^3}\geq\frac{c_5}{t}
		\quad(t<\tau)
	\end{align}
	for a deterministic $c_5>0$.

	\smallskip\noindent\textbf{Step 4. A bounded stopped statistic cannot have this drift forever.}
	For $t\geq n$, define
	\begin{align*}
		\cZ_t=
		\begin{cases}
			\norm{Y_{t\wedge\tau}}^2, & \text{on }A_n,   \\
			0,                     & \text{on }A_n^c.
		\end{cases}
	\end{align*}
	This definition evaluates $Y$ only where the initial bounds hold and
	up to the exit step, where Step 2 guarantees that the Schur matrix is
	still defined. Before exit $\norm{Y_t}\leq L$. If $\tau>n$ is finite,
	Step 2 also gives
	\begin{align*}
		\norm{Y_\tau}
        \leq
        \norm{Y_{\tau-1}}+
		\norm{Y_\tau-Y_{\tau-1}}
		\leq L+C_6/\gl_{\tau-1}\leq L+C_6
	\end{align*}
	for a deterministic $C_6$, since $\gl_{\tau-1}\geq1$.
	After exit the statistic is frozen. Consequently
	\begin{align}\label{eq:stopped-imbalance-bound}
		0\leq\cZ_t\leq C_{\mathrm{stop}}:=(L+C_6)^2
		\quad(t\geq n)
	\end{align}
	on the whole probability space. Because the stopped statistic retains
	its value at exit, control of this overshoot is needed for the uniform
	bound, even when exit is caused by a different gate.

	Since $\{t<\tau\}\subseteq A_n$, its increment is exactly
	\begin{align*}
		\gD\cZ_t=
		\begin{cases}
			\norm{Y_{t+1}}^2-\norm{Y_t}^2, & t<\tau,    \\
			0,                       & t\geq\tau.
		\end{cases}
	\end{align*}
	The first case includes $t+1=\tau$. The exit-step estimates were proved
	for every possible parent. Thus~\eqref{eq:Y-positive-drift} and the
	$\cF_t$-measurability of $\{t<\tau\}$ give
	\begin{align}\label{eq:stopped-imbalance-drift}
		\E_t\gD\cZ_t
		\geq\frac{c_5}{t}\ind_{\{t<\tau\}}.
	\end{align}
	Apply Lemma~\ref{lem:bounded-stopped-drift} with
	$c_t=c_5/t$, using~\eqref{eq:stopped-imbalance-bound} and
	\eqref{eq:stopped-imbalance-drift}. For every fixed choice of
	candidates, numerical bounds and starting time, this gives
	\begin{align}\label{eq:coexistence-region-exit}
		\pr(\tau=\infty)=0.
	\end{align}
	The application uses finite sums of conditional expectations under
	the original law, without conditioning on permanent survival.

	\smallskip\noindent\textbf{Step 5. Cover every possible coexistence event.}
	Exiting each fixed region almost surely rules out coexistence once we
	show that every coexistence path must eventually remain in some one
	fixed region. Fix a deterministic finite $I$ with $|I|\geq2$, and let
	$\Omega_*$ be the probability-one event on which the scale estimates
	and Lemma~\ref{lem:window} hold simultaneously for every deterministic
	finite candidate set.

	On $\cE_I\cap\Omega_*$, the growth scales have eventual positive
	lower and finite upper bounds, the exterior degrees are eventually
	$O(\gl_t)$, and the candidate centrality has positive lower limit.
	The competition-window lemma gives $U_t=O(\gl_t)$, and
	\eqref{eq:imbalance-size} then gives $Y_t=O(1)$. For each such path,
	choose positive rational $c,C,K,L$, with $c<C$, so that $c$ lies below
	the eventual lower bounds and $C,K,L$ lie above the corresponding
	eventual upper bounds. There is an integer $n_1$ after which
	all gates $\cD_t$ hold. Choosing an integer
	$n\geq\max\{n_1,n_0(I,c,C,K,L)\}$ produces a stopping time that never
	fires on this path.

	More formally, let $\sP_I$ be the set of tuples
	$\xi=(n,c,C,K,L)$ with positive rational $c,C,K,L$, $c<C$, and integer
	$n\geq n_0(I,c,C,K,L)$. Write $\tau^{I,\xi}$ for the associated
	stopping time. The preceding pathwise argument proves
	\begin{align}\label{eq:coexistence-countable-cover}
		\cE_I\cap\Omega_*
		\subseteq\bigcup_{\xi\in\sP_I}
		\{\tau^{I,\xi}=\infty\}.
	\end{align}
	Although the tuple that covers a particular path may depend on that
	path, all possible tuples lie in this fixed countable family. Each
	event on the right has probability zero by
	\eqref{eq:coexistence-region-exit}. Countable subadditivity and
	$\pr(\Omega_*)=1$ therefore give $\pr(\cE_I)=0$.

	There are also only countably many finite subsets of the integer
	vertex labels. Taking a second countable union yields
	$
    \pr\left(\cup_{2\leq |I|<\infty}
		\cE_I\right)=0.
	$
	Lemma~\ref{lem:active} supplies a nonempty finite active set almost
	surely, so that set must be a singleton $I=\{h\}$.

	The scalar Schur equation now gives $d_h(t)=\gl_t^2+O(\gl_t)$,
	and Lemma~\ref{lem:active} gives $\max_{u\ne h}d_u(t)=O(\gl_t)$.
	Thus $h$ is eventually the unique degree maximizer.
	For its Perron coordinate, define $z=v_h$ and $w=v_F$.
	Then $w=R(\gl_t)B^\top z$ and
	\begin{align*}
		\norm{w}^2=\gl_t^{-2}{d_h(t)}z^2+o(1)=z^2+o(1).
	\end{align*}
	Indeed, with $z=v_h(t)$ and scalar $BB^\top=d_h(t)$,
	\begin{align*}
		\norm{w}^2=\gl_t^{-2}{z^2}
		B(\mathrm{Id}-A_F/\gl_t)^{-2}B^\top
		=\gl_t^{-2}z^2d_h(t)
		\,(1+O(\norm{A_F}/\gl_t)).
	\end{align*}
	Here $\norm{A_F}/\gl_t=O(\gl_t^{-1/2})\to0$.
	Normalization gives
	$z\to2^{-1/2}$, and the Green bound gives
	$\max_{u\in F}v_u(t)=O(\gl_t^{-1})$.
\qed


\section{Shell balances and sharp limits}\label{sec:sharp}

For each deterministic proposed root $h$, begin at $n_h=\max\{t_0,h\}$ and define its shells $L_j^h(t)=\{u\in[t]:\dist_{T_t}(u,h)=j\}$. These sets are $\cF_t$-measurable, and distances of existing vertices are preserved by subsequent leaf additions. Thus shell-count drift identities are computed under the original attachment law for a fixed root. We prove the martingale estimates with deterministic bounds and adapted stopping times, and only afterward intersect with the event that the proposed root is the persistent hub. A countable union over roots, starting times, and rational bounds then gives the conclusions below.

For brevity, write $h$ for the hub supplied by Theorem~\ref{thm:hub}, and define $H_t=d_h(t)$, $m_t=v_h(t)$, and $L_j(t)=L_j^h(t)$.

\begin{thm}[Sharp scalar limits]\label{thm:sharp}
	Almost surely, with $\kappa_t=\sum_u v_u(t)^3/S_t$,
	\begin{align*}
		\frac{H_t}{\gl_t^2}\to1,\quad
		\frac{\gl_t^3}{t}\to\frac{\ga}{\theta(1+\theta)},\quad
		\frac{S_t}{\gl_t}\to\frac{1+\theta}{\sqrt2},\quad
		\gl_t\kappa_t\to\frac1{2(1+\theta)}.
	\end{align*}
\end{thm}
\begin{proof}
	Root the tree at $h$. Since $\max_{u\ne h}d_u(t)\le K\gl_t$
	eventually, leaf induction in the cavity recursion gives, uniformly over nonroot vertices,
	\begin{align}\label{eq:cavity}
		\gl_t^{-1}\le
		r_x:=\frac{v_x}{v_{p(x)}}
		=\left(\gl_t-\sum_{y:p(y)=x}r_y\right)^{-1}
		\le(\gl_t-2K)^{-1}
		=\gl_t^{-1}\,(1+O(\gl_t^{-1})).
	\end{align}
	Indeed, first induct with $r_y\le2/\gl_t$. Its denominator is
	then at least $\gl_t-2K$, which also proves the sharper estimate.
	Thus $v_x=m_t\gl_t^{-j}\,(1+O(\gl_t^{-1}))$ uniformly in
	$L_j$, $j=1,2$, and $v_x\le8m_t/\gl_t^3$ at distance at least
	three. The hub equation gives $H_t/\gl_t^2\to1$.
	The squared mass outside $L_0\cup L_1$ is $O(t/\gl_t^4)=o(1)$ by Proposition~\ref{prop:scales}, so normalization also gives $m_t\to1/\sqrt2$.

	Put $B_t=|L_2(t)|$ and $X_t=S_t/(m_t\gl_t)$. The exact equation
	$\sum_{u\in L_1}v_u=\gl_tm_t$ and the bound
	$\sum_{j\ge3}\sum_{u\in L_j}v_u=O(t/\gl_t^3)=O(1)$ give
	\begin{align}\label{eq:shellbalance}
		X_t=1+B_t/\gl_t^3+o(1)
		=1+B_t/H_t^{1/\theta}+o(1).
	\end{align}
	Moreover $X_t$ is bounded above and below by positive constants.
	For each fixed proposed root, the following are exact conditional drifts
	from time $n_h$ onward, before imposing any eventual hub event:
	\begin{align*}
		\E_t\gD H_t=\frac{\ga}{\gl_tX_t}+\frac\gb t,
		\quad
		\E_t\gD B_t=\frac\ga{X_t}+\frac{\gb H_t}t.
	\end{align*}
	Since a hub hit is the only event that increments $H_t$, Taylor's
	theorem and $1/\theta=3/2$ give
	\begin{align*}
		\gD(H_t^{1/\theta})
		=\left(\frac{1}\theta \sqrt{H_t}+O(H_t^{-1/2})\right)\gD H_t.
	\end{align*}
	Using $H_t\asymp t^{\theta}$ and $H_t/\gl_t^2\to1$, we obtain
	\begin{align*}
		\E_t\gD(H_t^{1/\theta})=\frac{\ga}{\theta X_t}+o(1),\quad
		\E_t\gD B_t=\frac\ga{X_t}+o(1).
	\end{align*}
	Both martingale errors are $o(t)$ almost surely. The conditional second
	moment for the first process is $O(H_t\E_t\gD H_t)=O(t^{\theta/2})$,
	and that for the second is bounded. On a fixed localization region,
	the normalized variance series are bounded by
	$C\sum_t t^{-5/3}$ and $C\sum_t t^{-2}$, respectively.
	Martingale convergence and Kronecker's lemma therefore apply. These estimates are applied
	with the stopping convention above when their constants are random.
	Subtracting $1/\theta$ times the second displayed drift from the
	first gives $o(1)$. Summing the corresponding compensators and
	martingale errors gives
	$H_t^{1/\theta}-B_t/\theta=o(t)$. Since
	$H_t^{1/\theta}\asymp t$, division by $H_t^{1/\theta}$ yields
	$B_t/H_t^{1/\theta}\to\theta$. Equation~\eqref{eq:shellbalance}
	then gives $X_t\to1+\theta$. Averaging the drift formulas
	now yields
	\begin{align*}
		B_t/t\to q:=\ga/(1+\theta),\quad
		H_t^{1/\theta}/t\to\ga/(\theta(1+\theta)).
	\end{align*}
	This proves the first three limits. Finally
	$\sum_{u\ne h}v_u^3\le\max_{u\ne h}v_u\sum_{u\ne h}v_u^2
		=O(\gl_t^{-1})$ by~\eqref{eq:cavity}. The hub term and the limit of
	$S_t/\gl_t$ give $\gl_t\kappa_t\to1/(2(1+\theta))$.
\end{proof}

\begin{rem}[Constants and centrality scales]\label{rem:constants}
Here $\theta=2/3$, $q=\ga/(1+\theta)$, and $b_2=\gb+\theta$. The balances in the preceding proof give the following limits.
\begin{center}
\begin{tabular}{ll@{\hspace{15ex}}ll}
\toprule
Quantity & Limit & Quantity & Limit \\
\midrule
$\gl_t^3/t$ & $\gc^3=\ga/(\theta(1+\theta))$ &
$|L_2(t)|/t$ & $q=\ga/(1+\theta)$ \\
$S_t/\gl_t$ & $(1+\theta)/\sqrt2$ &
$\gl_t\kappa_t$ & $1/(2(1+\theta))$ \\
$S_t^{-1}\sum_{L_1(t)}v_u(t)$ & $1/(1+\theta)$ &
$S_t^{-1}\sum_{L_2(t)}v_u(t)$ & $\theta/(1+\theta)$ \\
\end{tabular}
\end{center}
The second shell carries a positive fraction of the $\ell_1$ centrality but vanishing squared Perron mass. In each fixed degree class the first-shell centrality vanishes (Theorem~\ref{thm:profile}), so $\sum_{k\geq1}\lim_tA_k(t)=\theta/(1+\theta)<1$. The remaining centrality lies on degrees escaping to infinity.

For a second-shell vertex, the centrality part of the attachment probability is $\ga/(X_t\gl_t^3)\sim\theta/t$. Adding $\gb/t$ gives the parameter $b_2$. A fixed hub neighbor instead has attachment probability asymptotic to $\ga/((1+\theta)\gc^2t^\theta)$. Summation yields the coefficient
\begin{align*}
 \frac{\ga}{(1+\theta)\gc^2(1-\theta)}
 =\frac{\theta}{1-\theta}\gc=2\gc,
\end{align*}
which underlies Theorem~\ref{thm:second} and Corollary~\ref{cor:hierarchy}.
The inverse participation ratio satisfies $\sum_uv_u(t)^4\to1/4$, since the hub term tends to $1/4$ and
\begin{align*}
 \sum_{u\ne h}v_u(t)^4\leq
 \left(\max_{u\ne h}v_u(t)^2\right)\sum_{u\ne h}v_u(t)^2=o(1).
\end{align*}
This is a standard measure of eigenvector localization~\cite{MartinZhangNewman2014}.
\end{rem}

\begin{cor}[Distance from the persistent hub]\label{cor:distance}
	Almost surely, the empirical distance distribution
	$\mu_t(j)=|L_j(t)|/t$, $j\geq0$, converges in total variation to
	the probability distribution $\rho$ given by
	\begin{align*}
		\rho_0=\rho_1=0,\qquad
		\rho_2=q,\qquad \rho_3=qb_2,\qquad
		\rho_j=qb_2\gb^{j-3}\text{ for } j\geq4,
	\end{align*}
	where $q=\ga/(1+\theta)$ and $b_2=\gb+\theta$.
	If $U_t$ is sampled uniformly from $[t]$, independently of the tree
	growth, then the conditional law of $\dist_{T_t}(U_t,h)$ given
	$(T_t,h)$ has this limit almost surely, and its unconditional law
	has the same limit in total variation. When $\ga=1$, the limiting
	masses are $3/5$ at distance two and $2/5$ at distance three.
\end{cor}
\begin{proof}
	We first fix a deterministic proposed root $h$ and start at
	$n_h=\max\{t_0,h\}$. For $t\geq n_h$, write
	\begin{align*}
		L_j^h(t)&=\{u\in[t]:\dist_{T_t}(u,h)=j\},\qquad
		C_j(t)=|L_j^h(t)|,\qquad
		V_j(t)=\frac1{S_t}\sum_{u\in L_j^h(t)}v_u(t).
	\end{align*}
	Distances between existing vertices are preserved by later leaf
	additions. Consequently, for every $j\geq1$,
	\begin{align*}
		\gD C_j(t)&=\ind_{\{J_{t+1}\in L_{j-1}^h(t)\}},\text{ and }
		\E_t\gD C_j(t)=\ga V_{j-1}(t)+\gb C_{j-1}(t)/t.
	\end{align*}
	The centered increments are bounded martingale differences. Their
	partial sums are therefore $o(t)$ almost surely, simultaneously for
	all deterministic labels $h$ and all fixed $j$, by a countable
	intersection. These statements hold under the original attachment
	law, before intersecting with the event that $h$ is the persistent hub.

	On that event, Theorem~\ref{thm:sharp} gives
	$C_0(t)=1$, $C_1(t)=H_t=o(t)$, and $C_2(t)/t\to q$.
	The shell centrality limits in Remark~\ref{rem:constants} give
	\begin{align*}
		\E_t\gD C_3(t)
		\to \frac{\ga\theta}{1+\theta}+\gb q=qb_2.
	\end{align*}
	Ces\`aro averaging and the martingale bound yield
	$C_3(t)/t\to qb_2$. For $j\geq4$, the bound on the total centrality
	beyond the second shell in the proof of Theorem~\ref{thm:sharp}
	gives $V_{j-1}(t)\to0$. Induction in the same drift identity now
	gives $C_j(t)/t\to\gb\rho_{j-1}$. This proves all coordinate limits.

	If $0<\gb<1$, their sum is
	\begin{align*}
		\sum_{j\geq0}\rho_j
		=q+\frac{qb_2}{1-\gb}
		=\frac{\ga+\gb+\theta}{1+\theta}=1.
	\end{align*}
	When $\gb=0$, the recurrence instead gives directly
	$\rho_2=3/5$, $\rho_3=2/5$, and $\rho_j=0$ for $j\geq4$;
	again the masses sum to one. Coordinate convergence to a probability
	distribution on a countable set implies total-variation convergence:
	truncate at a fixed distance, pass to the limit on the finitely many
	coordinates, and then let the truncation increase.
	Conditional on $(T_t,h)$, uniform sampling has law $\mu_t$.
	For $t<h$, set the sampled distance to zero and $\mu_t=\delta_0$.
	The unconditional total-variation distance is at most
	$\E\abs{\mu_t-\rho}_{\mathrm{TV}}$, which tends to zero by bounded
	convergence. Since $h$ has a finite label almost surely, this convention
	does not affect any limiting assertion.
\end{proof}

We record two elementary tools used to pass from shell balances to
degree counts. If $X_s\in\{0,1\}$ has conditional mean $p_s$, set
$N_{r,t}=\sum_{s=r}^{t-1}X_s$ and
$\gL_{r,t}=\sum_{s=r}^{t-1}p_s$. For $\zeta>0$,
\begin{align}\label{eq:counting}
	\E e^{\zeta N_{r,t}-(e^\zeta-1)\gL_{r,t}}\le1,
	\quad
	\E e^{-\zeta N_{r,t}+(1-e^{-\zeta})\gL_{r,t}}\le1.
\end{align}
These follow by multiplying the one-step conditional exponential
bounds. They remain valid after stopping and bound joint events
involving the realized compensator.

\begin{lem}[Linear count recursion]\label{lem:recursion}
	If $a\ge0$ is fixed,
	$Y_{t+1}=(1-a/t)Y_t+f_t+\xi_{t+1}$, $f_t\to f$, and $\xi_{t+1}$
	are bounded martingale differences, then $Y_t/t\to f/(1+a)$ almost surely.
\end{lem}
\begin{proof}
	For $n>a$, write $P_{s,t}=\prod_{r=s}^{t-1}(1-a/r)$.
	Iteration expresses $Y_t$ as $P_{n,t}Y_n$ plus
	$\sum_{s=n}^{t-1}P_{s+1,t}(f_s+\xi_{s+1})$.
	These weights lie in $[0,1]$ and increase with $s$. Since the partial
	martingale sums $M_s$ are $o(s)$, summation by parts bounds their
	weighted sum by $2\max_{n\le s\le t}|M_s|=o(t)$.
	The identity $\sum_{s=n}^{t-1}P_{s+1,t}=(t-nP_{n,t})/(1+a)$
	handles constant forcing, and splitting $f_s-f$ at a fixed late index
	bounds its weighted sum by $o(t)$.
\end{proof}

\begin{thm}[Centrality profile]\label{thm:profile}
	For $A_k(t)=\sum_{u:d_u(t)=k}v_u(t)/S_t$ and every fixed $k\ge1$,
	almost surely
	\begin{align*}
		A_k(t)\to\frac{\theta}{1+\theta}\geom_{b_2}(k),\quad b_2=\gb+\theta.
	\end{align*}
\end{thm}
\begin{proof}
	Uniformly on the first shell,
	$a_t(u)\sim\ga/[(1+\theta)\gl_t^2]\asymp t^{-\theta}$.
	Fix $\eps\in(0,1)$, set $r=\lfloor(1-\eps)t\rfloor$, and
	consider a first-shell vertex born by $r$. Its compensator on $[r,t]$
	is at least $c_\eps t^{\theta/2}$ eventually, for a deterministic
	$c_\eps>0$. By~\eqref{eq:counting},
	\begin{align*}
		\pr\{N_{r,t}\le k,\ \gL_{r,t}\ge c_\eps t^{\theta/2}\}
		\le e^{k-(1-e^{-1})c_\eps t^{\theta/2}}.
	\end{align*}
	The union bound over the at most $t$ vertices is summable. Hence every
	degree-$k$ first-shell vertex is eventually born after $r$. If $C_k(t)$
	counts them, the sharp hub growth gives
	$\limsup C_k(t)/\gl_t^2\le1-(1-\eps)^{\theta}$.
	Letting rational $\eps\downarrow0$ yields
	$C_k(t)=o(\gl_t^2)$. Since $v_u/S_t=O(\gl_t^{-2})$ there,
	their degree-$k$ mass vanishes. For a fixed proposed hub $h$ and vertex $u$, use the selection
	indicators of $u$, multiplied by the $\cF_r$-measurable indicator
	that $u\in L_1(r)$. Thus the counting estimate applies to the joint
	event involving the realized compensator without conditioning on
	eventual hub success. Intersecting afterward with the event that $h$
	is the persistent hub proves the assertion, simultaneously for all
	possible fixed labels.

	Let $B_k(t)$ count degree-$k$ vertices in $L_2(t)$, so that
	$B_t=\sum_{k\geq1}B_k(t)$, and set $B_0(t)=0$.
	The scalar limits and~\eqref{eq:cavity} give uniformly on this shell
	\begin{align*}
		a_t(u)=\frac\gb t+
		\frac{\ga m_t}{\gl_t^2S_t}
		(1+O(\gl_t^{-1}))
		=\frac{b_2+o(1)}t.
	\end{align*}
	Births into this shell have probability tending to $q=\ga/(1+\theta)$.
	Thus
	$\E_t\gD B_k=q\ind_{\{k=1\}}+b_2(B_{k-1}-B_k)/t+o(1)$, with bounded martingale errors.
	For $k=1$, Lemma~\ref{lem:recursion} gives $B_1(t)/t\to q/(1+b_2)$.
	If $B_{k-1}(t)/t\to\rho_{k-1}$, the same lemma gives
	$B_k(t)/t\to b_2\rho_{k-1}/(1+b_2)$. Induction therefore yields
	$B_k(t)/t\to q\geom_{b_2}(k)$. The hub degree diverges, the first-shell
	contribution vanishes, and deeper shells have total normalized mass
	$O(1/S_t)=o(1)$. Therefore
	\begin{align*}
		A_k(t)=\frac{m_t}{\gl_t^2S_t}(1+o(1))\,B_k(t)+o(1)
		\to\frac{q\theta}{\ga}\geom_{b_2}(k)=\frac{\theta}{1+\theta}\geom_{b_2}(k).
	\end{align*}
	This completes the proof.
\end{proof}

\begin{thm}[Uniform-vertex degree law]\label{thm:degree}
	Almost surely, for every $k\ge1$,
	\begin{align*}
		\frac{N_k(t)}t\to
		\pi_k=q\geom_{b_2}(k)+(1-q)\geom_\gb(k),\quad
		q=\frac{\ga}{1+\theta},\quad b_2=\gb+\theta.
	\end{align*}
	The conditional law of the degree of an independent uniform vertex
	converges almost surely in total variation to $\pi$, and its
	unconditional law has the same limit. The limiting mean is $2-q$.
\end{thm}
\begin{proof}
	With $N_0=A_0=0$, the exact count drift is
	\begin{align*}
		\E_t\gD N_k=\ind_{\{k=1\}}
		+\frac\gb t(N_{k-1}-N_k)+\ga(A_{k-1}-A_k).
	\end{align*}
	The martingale errors are bounded. The preceding theorem and
	Lemma~\ref{lem:recursion} give limits satisfying
	$(1+\gb)\pi_k=\ind_{\{k=1\}}+\gb\pi_{k-1}
		+(\ga\theta/(1+\theta))(\geom_{b_2}(k-1)-\geom_{b_2}(k))$.
	The identity $(1+b)\geom_b(k)-b\geom_b(k-1)=\ind_{\{k=1\}}$ and
	$q(b_2-\gb)=\ga\theta/(1+\theta)$ verify the displayed solution.
	The displayed solution sums to one. To see the total-variation
	conclusion directly, define $\mu_t(k)=N_k(t)/t$. For every integer $k_0\geq1$,
	\begin{align*}
		\sum_{k\geq1}|\mu_t(k)-\pi_k|
		\leq\sum_{k=1}^{k_0}|\mu_t(k)-\pi_k|
		+\mu_t(\{k>k_0\})+\pi(\{k>k_0\}).
	\end{align*}
	First let $t\to\infty$ and then $k_0\to\infty$. The right side tends
	to zero. For the unconditional law, its total-variation distance from
	$\pi$ is at most $\E\abs{\mu_t-\pi}_{\mathrm{TV}}$, which tends to
	zero by bounded convergence.
	Its mean is $1+\gb+\ga\theta/(1+\theta)=2-q$, whereas the finite-tree mean is
	$2-2/t$. The missing first moment $q$ is carried by the $H_t=o(t)$
	first-shell vertices, whose total degree is $H_t+B_t\sim qt$.
\end{proof}

\begin{thm}[Second-largest degree]\label{thm:second}
	Almost surely, $D^{\star\star}_t/\gl_t\to2$.
\end{thm}
\begin{proof}
	Put $a=\ga/((1+\theta)\gc^2)$, so that $a/(1-\theta)=2\gc$. Uniformly in $u\ne h$, the cavity estimates give
	$a_t(u)\le(1+o(1))a t^{-\theta}$, with asymptotic equality for every
	neighbor of $h$. To see the uniformity, the first-shell coordinates
	equal $m_t\gl_t^{-1}\{1+O(\gl_t^{-1})\}$ uniformly, and
	every cavity ratio is eventually less than one. Coordinates at
	greater distances are therefore no larger than their first-shell
	ancestor's coordinate. The uniform term $\gb/t$ is
	$o(t^{-\theta})$.
	For fixed labels $h$, a fixed $\eta>0$ and a deterministic $n$, stop
	at the first failure of the envelope $(1+\eta)a s^{-\theta}$ for any
	existing nonhub vertex. Starting at $\max(n,u)$, the stopped selection
	count of each $u\ne h$ has exponential moment at most
	\begin{align*}
		\exp\Big((e^\zeta-1)(1+\eta)2\gc t^{\theta/2}+O(1)\Big).
	\end{align*}
	Its initial degree is at most $n+1$. For any $\eps>0$, choose
	$\eta$ and then $\zeta>0$ small. Markov's inequality and the union
	bound give probability at most $t e^{-c_\eps t^{\theta/2}}$ that any
	nonhub degree exceeds $(2\gc+\eps)t^{\theta/2}$ before failure.
	This is summable. Every successful hub path has a nonfailing $n$ for
	each $\eta$, so the countable union over $h,n$ and then
	$\eps\downarrow0$ gives the upper limit $2\gc$.

	Choose a neighbor $j$ already adjacent to $h$ when $h$ is born,
	or a seed neighbor. Its selection compensator satisfies
	$$\gL_j(t)\sim a\sum_{s<t}s^{-\theta}\sim2\gc t^{\theta/2}.$$
	For each fixed pair of labels $(h,j)$, count selections of $j$ from
	the deterministic time when both labels are present. The counting
	inequality applies to its realized compensator under the original
	attachment law. Only afterward do we intersect with the event that
	$h$ is the hub and $j$ is its neighbor, on which the displayed
	compensator asymptotic holds.
	For small $\zeta>0$, the second exponential inequality in~\eqref{eq:counting} shows that
	the joint event
	$$
    \{N_j(t)<(2\gc-2\eps)t^{\theta/2},
		\gL_j(t)\ge(2\gc-\eps)t^{\theta/2}\}
    $$
	has probability at most $e^{-c_\eps t^{\theta/2}}$.
	Applying the Borel--Cantelli lemma and taking a countable union over fixed pairs $(h,j)$
	gives the lower limit $2\gc$. Finally, $H_t\sim\gl_t^2$ dominates
	all nonhub degrees, so their maximum is eventually $D^{\star\star}_t$.
\end{proof}

\begin{cor}[Degrees beyond the first shell]\label{cor:hierarchy}
	Almost surely,
	\begin{align*}
		\max_{u\,:\,\dist_{T_t}(u,h)\geq2}d_u(t)=O(\log t).
	\end{align*}
	Moreover, simultaneously for every fixed vertex $j$ adjacent to $h$
	in the infinite growing tree,
	\begin{align*}
		{d_j(t)}/{\gl_t}\to 2.
	\end{align*}
\end{cor}
\begin{proof}
\smallskip\noindent\textbf{A uniform envelope beyond the first shell.}
	By Theorem~\ref{thm:sharp} and~\eqref{eq:cavity}, uniformly for
	$u\in L_2(t)$, we have
	\begin{align*}
		\frac{\ga v_u(t)}{S_t}
		=\frac{\ga}{(1+\theta)\gl_t^3}(1+o(1))
		=\frac{\theta+o(1)}t.
	\end{align*}
	Every cavity ratio is eventually less than one, so a vertex farther
	from $h$ has centrality no greater than that of its second-shell
	ancestor. Adding the uniform attachment term gives
	\begin{align}\label{eq:far-envelope}
		a_t(u)\leq(b_2+o(1))/t
		\text{ uniformly for }
        \dist_{T_t}(u,h)\geq2.
	\end{align}

\smallskip\noindent\textbf{An adapted logarithmic bound.}
	Fix a deterministic proposed root $h$, a rational constant $C>b_2$,
	and an integer $n\geq\max\{t_0,h\}$. Let $\tau$ be the first
	$s\geq n$ at which an existing vertex at distance at least two from
	$h$ has $a_s(u)>C/s$. This is a stopping time. For a deterministic
	label $u$, define
	\begin{align*}
		X_t(u)=\sum_{s=n}^{t-1}
		\ind_{\{s<\tau\}}\ind_{\{u\leq s\}}
		\ind_{\{\dist_{T_s}(u,h)\geq2\}}
		\ind_{\{J_{s+1}=u\}},
	\end{align*}
	where the distance indicator is taken to be zero when $u>s$.
	Each increment has conditional mean at most $C/s$. The one-step
	exponential bound underlying~\eqref{eq:counting} therefore yields,
	for every $\zeta>0$,
	\begin{align*}
		\E e^{\zeta X_t(u)}
		&\leq\exp\left(C(e^\zeta-1)
		\sum_{s=n}^{t-1}\frac1s\right)
		\leq K_{n,C,\zeta}\,t^{C(e^\zeta-1)},
	\end{align*}
	uniformly over labels $u\leq t$. Markov's inequality and a union
	bound give
	\begin{align*}
		\pr\left\{\max_{u\leq t}X_t(u)>A\log t\right\}
		\leq K_{n,C,\zeta}\,
		t^{1+C(e^\zeta-1)-\zeta A}.
	\end{align*}
	Choose $A$ with $\zeta A>C(e^\zeta-1)+2$. The bound is summable
	in $t$, so Borel--Cantelli lemma gives $\max_{u\leq t}X_t(u)=O(\log t)$
	almost surely. On $\{\tau=\infty\}$, every vertex at distance at
	least two has degree at most $n+X_t(u)$: older vertices have degree
	at most $n-1$ at time $n$, while later vertices have degree one at
	birth, and their distance from $h$ is then fixed.

	Take a countable intersection over the deterministic labels $h$,
	starting times $n$, and rational constants $C>b_2$. On every path
	where $h$ is the persistent hub,~\eqref{eq:far-envelope} ensures a
	starting time for which $\tau=\infty$. This proves the logarithmic
	bound under the original law, without conditioning a transition
	probability on eventual hub success. The argument includes $\gb=0$,
	when $b_2=\theta$.

\smallskip\noindent\textbf{The degree of each fixed hub neighbor.}
	Fix deterministic labels $(h,j)$ with $h\ne j$ and put
	$n_{h,j}=\max\{t_0,h,j\}$. Count selections and their compensator
	over exactly the same interval
	\begin{align*}
		N_j(t)&=\sum_{s=n_{h,j}}^{t-1}\ind_{\{J_{s+1}=j\}},\qquad
		\gL_j(t)=\sum_{s=n_{h,j}}^{t-1}a_s(j).
	\end{align*}
	On the event that $h$ is the persistent hub and $j$ is its neighbor,
	the cavity estimates and sharp scalar limits imply
	\begin{align*}
		a_s(j)&\sim
		\frac{\ga}{(1+\theta)\gc^2}s^{-\theta},\qquad
		\gL_j(t)\sim
		\frac{\ga}{(1+\theta)\gc^2(1-\theta)}t^{1-\theta}
		=2\gc t^{\theta/2}.
	\end{align*}
	Here $\gb/s=o(s^{-\theta})$, including when $\gb=0$.
	For $0<\eps<\gc$ and a sufficiently small $\zeta>0$,
	\eqref{eq:counting} bounds each of the joint events
	\begin{align*}
		\big\{N_j(t)&<(2\gc-2\eps)t^{\theta/2},
		\gL_j(t)\geq(2\gc-\eps)t^{\theta/2}\big\},\quad 
		\big\{N_j(t)>(2\gc+2\eps)t^{\theta/2},
		\gL_j(t)\leq(2\gc+\eps)t^{\theta/2}\big\}
	\end{align*}
	by $e^{-c_\eps t^{\theta/2}}$ for some $c_\eps>0$.
	Indeed, the exponents in the lower- and upper-tail bounds have
	respective positive coefficients
	\begin{align*}
		&(1-e^{-\zeta})(2\gc-\eps)
		-\zeta(2\gc-2\eps),\text{ and }
		\zeta(2\gc+2\eps)
		-(e^\zeta-1)(2\gc+\eps),
	\end{align*}
	each equal to $\zeta\eps+O(\zeta^2)$ as $\zeta\downarrow0$.
	Both bounds are summable. Borel--Cantelli lemma is applied before
	intersecting with the hub/neighbor event, on which the compensator
	asymptotic holds. Taking a countable intersection over fixed pairs
	and rational $\eps\downarrow0$ gives
	$N_j(t)/t^{\theta/2}\to2\gc$ for every fixed hub neighbor.
	Since $d_j(t)=d_j(n_{h,j})+N_j(t)$ and
	$\gl_t\sim\gc t^{\theta/2}$, this proves the second assertion.

	Finally, $h$ has a neighbor with a fixed finite label. Its degree is
	asymptotic to $2\gl_t$, whereas $\gl_t\gg\log t$. Together with
	the logarithmic bound, this shows that every nonhub degree maximizer
	eventually belongs to $L_1(t)$. 
\end{proof}

\bibliographystyle{alpha}
\bibliography{epan.bib}

@article{Bonacich1972,
  author={Bonacich, Phillip},
  title={Factoring and weighting approaches to status scores and clique identification},
  journal={The Journal of Mathematical Sociology},
  volume={2}, number={1}, pages={113--120}, year={1972},
  doi={10.1080/0022250X.1972.9989806},
  note={\href{https://doi.org/10.1080/0022250X.1972.9989806}{doi:10.1080/0022250X.1972.9989806}}
}

@article{BarabasiAlbert1999,
  author={Barab{\'a}si, Albert-L{\'a}szl{\'o} and Albert, R{\'e}ka},
  title={Emergence of Scaling in Random Networks},
  journal={Science}, volume={286}, number={5439},
  pages={509--512}, year={1999},
  doi={10.1126/science.286.5439.509},
  note={\href{https://doi.org/10.1126/science.286.5439.509}{doi:10.1126/science.286.5439.509}}
}

@article{Adami2024,
  author={Adami, Vadood and Ebadi, Zahra and Nattagh-Najafi, Morteza},
  title={A dandelion structure of eigenvector preferential attachment networks},
  journal={Scientific Reports}, volume={14}, pages={16994}, year={2024},
  doi={10.1038/s41598-024-67896-9},
  note={\href{https://doi.org/10.1038/s41598-024-67896-9}{doi:10.1038/s41598-024-67896-9}}
}

@article{Galashin2016,
  author={Galashin, Pavel},
  title={Existence of a persistent hub in the convex preferential attachment model},
  journal={Probability and Mathematical Statistics},
  volume={36}, number={1}, pages={59--74}, year={2016},
  eprint={1310.7513}, archivePrefix={arXiv},
  url={https://arxiv.org/abs/1310.7513},
  note={\href{https://arxiv.org/abs/1310.7513}{arXiv:1310.7513}}
}

@article{BanerjeeBhamidi2021,
  author={Banerjee, Sayan and Bhamidi, Shankar},
  title={Persistence of hubs in growing random networks},
  journal={Probability Theory and Related Fields},
  volume={180}, number={3--4}, pages={891--953}, year={2021},
  doi={10.1007/s00440-021-01066-0},
  note={\href{https://doi.org/10.1007/s00440-021-01066-0}{doi:10.1007/s00440-021-01066-0}}
}

@article{MartinZhangNewman2014,
  author = {Martin, Travis and Zhang, Xiao and Newman, M. E. J.},
  title = {Localization and centrality in networks},
  journal = {Physical Review E},
  volume = {90}, number = {5}, pages = {052808}, year = {2014},
  doi = {10.1103/PhysRevE.90.052808},
  url = {https://arxiv.org/abs/1401.5093},
  note = {\href{https://doi.org/10.1103/PhysRevE.90.052808}{doi:10.1103/PhysRevE.90.052808}}
}

@article{Pemantle1990,
  author = {Pemantle, Robin},
  title = {Nonconvergence to Unstable Points in Urn Models and Stochastic Approximations},
  journal = {The Annals of Probability},
  volume = {18}, number = {2}, pages = {698--712}, year = {1990},
  doi = {10.1214/aop/1176990853},
  url = {https://www2.math.upenn.edu/~pemantle/papers/nonconvergence.pdf},
  note = {\href{https://doi.org/10.1214/aop/1176990853}{doi:10.1214/aop/1176990853}}
}

@article{Benaim1999,
  author = {Bena{\"i}m, Michel},
  title = {Dynamics of stochastic approximation algorithms},
  journal = {S{\'e}minaire de Probabilit{\'e}s},
  volume = {33}, pages = {1--68}, year = {1999},
  url = {https://numdam.org/item/SPS_1999__33__1_0/},
  note = {\href{https://numdam.org/item/SPS_1999__33__1_0/}{Numdam record}; Lecture Notes in Mathematics 1709}
}

@article{JogLoh2018,
  author = {Jog, Varun and Loh, Po-Ling},
  title = {Persistence of centrality in random growing trees},
  journal = {Random Structures \& Algorithms},
  volume = {52}, number = {1}, pages = {136--157}, year = {2018},
  doi = {10.1002/rsa.20726},
  url = {https://arxiv.org/abs/1511.01975},
  note = {\href{https://doi.org/10.1002/rsa.20726}{doi:10.1002/rsa.20726}; first published online in 2017}
}

@article{RudasTothValko2007,
  author = {Rudas, Anna and T{\'o}th, B{\'a}lint and Valk{\'o}, Benedek},
  title = {Random trees and general branching processes},
  journal = {Random Structures \& Algorithms},
  volume = {31}, number = {2}, pages = {186--202}, year = {2007},
  doi = {10.1002/rsa.20137},
  url = {https://arxiv.org/abs/math/0503728},
  note = {\href{https://doi.org/10.1002/rsa.20137}{doi:10.1002/rsa.20137}}
}

@book{Williams1991,
  author = {Williams, David},
  title = {Probability with Martingales},
  publisher = {Cambridge University Press},
  series = {Cambridge Mathematical Textbooks},
  year = {1991},
  isbn = {9780521406055},
  doi = {10.1017/CBO9780511813658},
  note = {\href{https://doi.org/10.1017/CBO9780511813658}{doi:10.1017/CBO9780511813658}}
}

@article{BanerjeeBhamidi2022,
  author = {Banerjee, Sayan and Bhamidi, Shankar},
  title = {Root finding algorithms and persistence of {Jordan} centrality in growing random trees},
  journal = {The Annals of Applied Probability},
  volume = {32}, number = {3}, pages = {2180--2210}, year = {2022},
  doi = {10.1214/21-AAP1731},
  url = {https://arxiv.org/abs/2006.15609},
  note = {\href{https://doi.org/10.1214/21-AAP1731}{doi:10.1214/21-AAP1731}}
}
\end{document}